\documentclass{article}
\usepackage{amsfonts,amsmath,amsthm,latexsym,hyperref} 
\usepackage{xcolor}
\usepackage{tikz}
\usepackage{lscape} 
\newcounter{satznum}
\newtheorem{theorem}{Theorem}[satznum]

\newcounter{cornum}
\newtheorem{corollary}{Corollary}[cornum]

\newcounter{defnum}
\newtheorem{definition}{Definition}[defnum]

\newcounter{lemmanum}
\newtheorem{lemma}{Lemma}[lemmanum]

\newcounter{propnum}
\newtheorem{proposition}{Proposition}[propnum]

\newenvironment{remark}
 {\begin{trivlist}\item[]{\bf Remark.}}
 {\end{trivlist}}

{\makeatletter
\gdef\me{{\mathbb E}} 
\gdef\nz{{\mathbb N}} 
\gdef\pr{{\mathbb P}} 
\gdef\rz{{\mathbb R}} 
\gdef\gz{{\mathbb Z}} 
}
\newcounter{todocounter}

\makeatletter
\def\@MRExtract#1 #2!{#1}
\newcommand{\MR}[1]{
  \xdef\@MRSTRIP{\@MRExtract#1 !}
  \href{http://www.ams.org/mathscinet-getitem?mr=\@MRSTRIP}{MR\@MRSTRIP}}
\makeatother
\begin{document}
   \section*{ON THE GENEALOGY OF MULTI-TYPE CANNINGS MODELS AND THEIR LIMITING EXCHANGEABLE COALESCENTS}
   {\sc Maximilian Flamm and Martin M\"ohle}\footnote{Fachbereich Mathematik, Eberhard Karls Universit\"at T\"ubingen,
   Auf der Morgenstelle 10, 72076 T\"ubingen, Germany\\ E-mail addresses: maximilian.flamm@uni-tuebingen.de, martin.moehle@uni-tuebingen.de}
\begin{center}
   Date: \today\\
\end{center}
\begin{abstract}
   We study the multi-type Cannings population model. Each individual has a type belonging to a given at most countable type space $E$. The population is hence divided into $|E|$ subpopulations. The subpopulation sizes are assumed to be constant over the generations, whereas the number of offspring of type $\ell\in E$ of all individuals of type $k\in E$ is allowed to be random. Under a joint exchangeability assumption on the offspring numbers, the transition probabilities of the ancestral process of a sample of individuals satisfy a multi-type consistency property, paving a way to prove in the limit for large subpopulation sizes the existence of multi-type exchangeable coalescent processes via Kolmogorov's extension theorem. Integral representations for the infinitesimal rates of these multi-type exchangeable coalescents and some of their properties are studied. Examples are provided, among them multi-type Wright--Fisher models and multi-type pure mutation models. The results contribute to the foundations of multi-type coalescent theory and provide new insights into (the existence of) multi-type exchangeable coalescents.

   \vspace{2mm}

   \noindent Keywords: Consistency; exchangeability; integral representation; multi-type Cannings model; multi-type coalescent

   \vspace{2mm}

   \noindent 2020 Mathematics Subject Classification:
            Primary 60J90; 
            60J10 
            Secondary 92D15; 
            92D25 
\end{abstract}
\subsection{Introduction and model definition} \label{intro}
   Multi-type population models in the spirit of Cannings \cite{Cannings1974,Cannings1975,Cannings1976} are studied, where each individual is equipped with a certain type taken from a given type space $E$. It is assumed that $E$ is at most countable. The population is hence divided into $|E|$ subpopulations. It is furthermore assumed that the size $N_k\in\nz:=\{1,2,\ldots\}$ of each subpopulation $k\in E$ is deterministic and constant over the generations. We denote by $N:=(N_k)_{k\in E}\in\nz^E$ the vector of all subpopulation sizes. In the following the dynamics of the model is first described for a single generation step. For $k,\ell\in E$ and $i\in[N_k]:=\{1,\ldots,N_k\}$ let $\nu_{k,\ell,i}$ denote the (random) number of offspring of type $\ell\in E$ of the $i$-th individual of type $k\in E$. Clearly, $N_{k,\ell}:=\sum_{i\in[N_k]}\nu_{k,\ell,i}$ is the number of offspring of type $\ell\in E$ of all individuals of type $k\in E$. Note that $\sum_{k\in E}N_{k,\ell}$ is the total number of offspring of type $\ell\in E$, whereas $\sum_{\ell\in E}N_{k,\ell}$ is the number of offspring of all individuals of type $k\in E$. The assumption that all subpopulation sizes are constant over time puts the restrictions
   \begin{equation} \label{size1}
      \sum_{k\in E}N_{k,\ell}\ =\ N_\ell,\qquad \ell\in E,
   \end{equation}
   on the random variables $N_{k,\ell}$, $k,\ell\in E$. Eq.~(\ref{size1}) is the analog of Cannings \cite[Eq.~(30)]{Cannings1975} and H\"ossjer \cite[Eq.~(2.1)]{Hoessjer2011}. The subclass of models satisfying $N_{k,k}=N_k$ for all $k\in E$ or, equivalently, $N_{k,\ell}=0$ for all $k,\ell\in E$ with $k\ne\ell$, is studied in \cite{Moehle2024}. Our analysis is based on the assumption (compare also with \cite[p.~266]{Cannings1975}) that the offspring sizes are jointly exchangeable, i.e.,
   \begin{itemize}
      \item[(A)] for all permutations
         $\pi_{k,\ell}$ of $[N_k]$, $k,\ell\in E$, $(\nu_{k,\ell,\pi_{k,\ell}(i)})_{k,\ell\in E,i\in[N_k]}$ has the same distribution as $(\nu_{k,\ell,i})_{k,\ell\in E, i\in[N_k]}$.
   \end{itemize}
   Less general, as in \cite{Hoessjer2011}, one may assume that the offspring sizes are
   \begin{itemize}
      \item[(A1)] exchangeable in each subpopulation, i.e., for each $k\in E$ the random vectors $\nu_{k,i}:=(\nu_{k,\ell,i})_{\ell\in E}$, $i\in[N_k]$, are exchangeable and
      \item[(A2)] independent in different subpopulations, i.e., the $(\nu_{k,i})_{i\in[N_k]}$, $k\in E$, are independent.
   \end{itemize}
   Clearly, (A) implies (A1). If (A1) and (A2) are satisfied, then (A) holds as well, since independence implies exchangeability. If (A2) is satisfied, then, for each $\ell\in E$, the $N_{k,\ell}$, $k\in E$, are independent, which together with (\ref{size1}) implies that the $N_{k,\ell}$ are deterministic almost surely, which is the reason why in \cite{Hoessjer2011} it is assumed that the $N_{k,\ell}$ are deterministic. Our results however hold under the more general assumption (A). In particular, the $N_{k,\ell}$ are allowed to be random and all results in this article hold for random $N_{k,\ell}$ satisfying (\ref{size1}).

   Up to now the model is defined for a single generation step. The description of the model over the generations is rather simple. Offspring sizes in different generations, labeled with $r\in\gz:=\{\ldots,-1,0,1,\ldots\}$, are simply assumed to be independent and identically distributed (iid).

   Figure \ref{figure1} illustrates a realization of two consecutive generations of the model with type space $E=\{1,2,3\}$, subpopulation sizes $N=(N_1,N_2,N_3)=(4,6,5)$ and with the realization
   \[
   (N_{k,\ell}(\omega))_{k,\ell\in E}=
      \left(
      \begin{tabular}{ccc}
         3 & 2 & 1\\
         1 & 4 & 3\\
         0  & 0 & 1
      \end{tabular}
   \right).
   \]
   The offspring numbers can be read off from Figure \ref{figure1}. For example, the first individual of the first (blue) subpopulation has $\nu_{1,1,1}(\omega)=2$ offspring of type $1$ (blue), $\nu_{1,2,1}(\omega)=1$ offspring of type $2$ (green) and $\nu_{1,3,1}(\omega)=0$ offspring of type $3$ (red).
   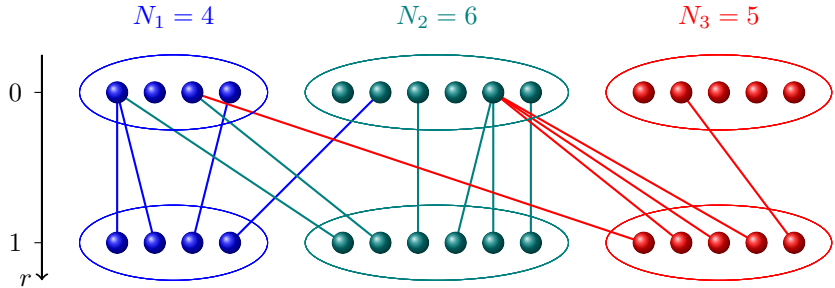
\begin{figure}[htpb]
      \caption[]{An illustration of two consecutive generations with three subpopulations}
      \centering
      \definecolor{myblue}{HTML}{92dcec}
\begin{tikzpicture}[scale=1.0]
   \def \N {10} 
   \def \r {1}
   \def \radius {0.15} 
   
   \draw[->,thick] (-0.5,2*\r+0.5) -- (-0.5,0-0.5) node[below,left] {$r$};
   
   \foreach \y in {0,...,\r}
      \draw (-0.6,2*\r-2*\y) -- (-0.5,2*\r-2*\y) node[left=4pt] {\y};
   
   \draw[thick,color=blue] (0.5,0) -- (0.5,2);
   \draw[thick,color=blue] (1,0) -- (0.5,2);
   \draw[thick,color=blue] (1.5,0) -- (2,2);
   \draw[thick,color=blue] (2,0) -- (4,2);
   
   \draw[thick,color=teal] (3.5,0) -- (0.5,2);
   \draw[thick,color=teal] (4,0) -- (1.5,2);
   \draw[thick,color=teal] (4.5,0) -- (4.5,2);
   \draw[thick,color=teal] (5,0) -- (5.5,2);
   \draw[thick,color=teal] (5.5,0) -- (5.5,2);
   \draw[thick,color=teal] (6,0) -- (6,2);

   \draw[thick,color=red] (7.5,0) -- (1.5,2);
   \draw[thick,color=red] (8,0) -- (5.5,2);
   \draw[thick,color=red] (8.5,0) -- (5.5,2);
   \draw[thick,color=red] (9,0) -- (5.5,2);
   \draw[thick,color=red] (9.5,0) -- (8,2);

   \draw[color=blue] (1.25,2*\r+1) node {\centering $N_1=4$};
   \foreach \n in {0,...,\r} \foreach \i in {1,...,4}
      {
      \draw[color=blue] (1.25,2*\n) ellipse (1.25cm and 5mm);
      \shade[ball color=blue] (\i/2,2*\n) circle (\radius);
      }
   \draw[color=teal] (4.75,2*\r+1) node {\centering $N_2=6$};
   \foreach \n in {0,...,\r} \foreach \i in {1,...,6}
      {
      \draw[color=teal] (4.75,2*\n) ellipse (1.75cm and 5mm);
      \shade[ball color=teal] (3+\i/2,2*\n) circle (\radius);
      }
   \draw[color=red] (8.5,2*\r+1) node {\centering $N_3=5$};
   \foreach \n in {0,...,\r} \foreach \i in {1,...,5}
      {
      \draw[color=red] (8.5,2*\n) ellipse (1.5cm and 5mm);
      \shade[ball color=red] (7+\i/2,2*\n) circle (\radius);
      }
\end{tikzpicture}
      \label{figure1}
   \end{figure}
   For $|E|=1$ and under (A), this model reduces to the classical single-type exchangeable population model introduced by Cannings \cite{Cannings1974, Cannings1975}.

   The paper is organized as follows. In Section \ref{ancestry} basic properties on the ancestry of multi-type Cannings models are derived. Section \ref{consistency} provides new multi-type consistency equations and discusses symmetry and exchangeability issues of multi-type Cannings models. Convergence results as all subpopulation sizes become large are provided in Section \ref{limits}. Sections \ref{coal} and \ref{intrep} deal with (the existence of) multi-type exchangeable coalescent processes and their integral representations. Two examples, the multi-type Wright--Fisher model and a multi-type pure mutation model, are extensively analysed in Section \ref{examples}. The proofs are provided in Sections \ref{proofs} except for the proof of the integral representation (Theorem \ref{main3}), which is -- because of its length -- provided in the separate Section \ref{proofmain3}. The article finishes with short appendix collecting some basic bounds, definitions, illustrations and results used throughout the article. The results extend those obtained in \cite{Moehle2024} and, hence, contribute to the foundations of multi-type coalescent theory.

   It is almost impossible to mention all relevant literature here. Precise citations are made at the appropriate places throughout the article. For recent related works on multi-type $\Lambda$-coalescents we refer the reader exemplary to Gonz\'alez et al. \cite{GonzalezCasanovaKurtNunezMoralesPerez2024} and Johnston, Kyprianou and Rogers \cite{JohnstonKyprianouRogers2023}.

\subsection{Ancestral structure} \label{ancestry}
   Before we will define the multi-type ancestral process, let us recall some basics on set partitions having labeled blocks, also called labeled set partitions. We shall also introduce some notations for random labeled partitions.
\subsubsection{Block labeled set partitions}
   The literature on labeled set partitions, also called colored, marked or typed set partitions, is sparse compared to the immense literature on standard set partitions. We refer the reader to Kallenberg \cite{Kallenberg2005} for probabilistic theory on marked partitions, to Goyt and Pudwell \cite{GoytPudwell2011,GoytPudwell2012} for some literature on the combinatorics of colored set partitions and to Alberti \cite{Alberti2024}, where a labeled partitioning process is studied. For $n\in\nz$ let ${\cal P}_n$ denote the space of partitions of $[n]$. Note that $|{\cal P}_n|=\sum_{j=1}^nS(n,j)$, where the $S(.,.)$ denote the Stirling numbers of the second kind, i.e., $S(n,j)$ is the number of partitions of $[n]$ having $j$ blocks. Any partition $\pi\in{\cal P}_n$ can be written as $\pi=\{B_1,\ldots,B_j\}$, where $B_1,\ldots,B_j$ are the (non-empty) blocks of $\pi$. The order of the blocks is not relevant, but usually the blocks are listed in order of appearance, i.e.~$1\in B_1$, $\min([n]\setminus B_1)\in B_2$ and so on. Given a partition $\{B_1,\ldots,B_j\}\in{\cal P}_n$ having $j$ blocks $B_1,\ldots,B_j$, one may equip each block $B_i$ with a type $k_i\in E$ leading to the partition $\{(B_1,k_1),\ldots,(B_j,k_j)\}$ having labeled (colored) blocks. Such a partition is called a block-labeled partition or simply a labeled (or colored, marked or typed) partition. Let ${\cal P}_{n,E}$ denotes the space of labeled partitions of $[n]$. If $d:=|E|<\infty$, then $|{\cal P}_{n,E}|=\sum_{j=1}^n d^j S(n,j)<\infty$. Alternatively, by Dobi\'nski's formula (see, for example, Eq.~(21) of Hsu and Shiue \cite{HsuShiue1998} or Mansour \cite[p.~384, Example 8.12]{Mansour2013}), $|{\cal P}_{n,E}|=e^{-d}\sum_{j\ge 0}d^jj^n/j!$ if $d<\infty$. For $\pi\in{\cal P}_{n,E}$ any block $(B,k)\in\pi$ is called a $k$-block of $\pi$.

   It is sometimes useful to delete all labels of a labeled partition. Formally, for $n\in\nz$, the function $r_n:{\cal P}_{n,E}\to{\cal P}_n$, defined via $r_n(\pi):=\{B_1,\ldots,B_j\}$ for all $\pi=\{(B_1,k_1),\ldots,(B_j,k_j)\}\in{\cal P}_{n,E}$, is called the ($n$-th) label removal function. The function $r_n$ simply maps each labeled partition to its non-labeled counterpart by removing all labels.

   For a permutation $\sigma\in S_n$ and a block $B\subseteq[n]$ define $\sigma(B):=\{\sigma(i):i\in B\}$ and for $\pi=\{(B_1,k_1),\ldots,(B_j,k_j)\}\in{\cal P}_{n,E}$ define $\sigma(\pi):=\{(\sigma(B_1),k_1),\ldots,(\sigma(B_j),k_j)\}\in{\cal P}_{n,E}$.
\subsubsection{Random labeled partitions}
   A random labeled partition of $[n]$ (with label space $E$) is a random variable $\Pi$ taking values in ${\cal P}_{n,E}$. The following notions of exchangeability are in the spirit of Aldous \cite{Aldous1985}, Kallenberg \cite[p.~343]{Kallenberg2005} and Pitman \cite{Pitman1995}. For a probability space $(\Omega,{\cal F},\pr)$ and $A\in{\cal F}$ let $\pr_A$ denote the restriction of $\pr$ to $A$ defined via $\pr_A(B):=\pr(B\cap A)$ for all $B\in{\cal F}$.
   \begin{definition}[Partial exchangeability] \label{partialexchangeability}
      Let $n\in\nz$, $G\subseteq S_n$ and $A\in{\cal F}$. A random labeled partition $\Pi$ of $[n]$ is called partially exchangeable with respect to $G$, if $\sigma(\Pi)\stackrel{d}{=}\Pi$ for all $\sigma\in G$, i.e., $\pr(\sigma(\Pi)=\pi)=\pr(\Pi=\pi)$ for all $\sigma\in G$ and all $\pi\in{\cal P}_{n,E}$. A random labeled partition $\Pi$ is called partially exchangeable with respect to $G$ on the event $A$, if $\pr_A(\sigma(\Pi)=\pi)=\pr_A(\Pi=\pi)$ for all $\sigma\in G$ and all $\pi\in{\cal P}_{n,E}$, where $\pr_A$ denotes the restriction of $\pr$ to $A$.
   \end{definition}
   \begin{remark}
      For $A=\Omega$, partial exchangeability with respect to $G$ on $A$ is the same as partial exchangeability with respect to $G$. If $\Pi$ is partially exchangeable with respect to $G$, then $\Pi$ is even partially exchangeable with respect to the group $\langle G\rangle$ generated by $G$. Partial exchangeability with respect to the full group $G=S_n$ of permutations of $[n]$ is simply called exchangeability. We will come back to exchangeability issues at the end of Section \ref{consistency}.
   \end{remark}

\subsubsection{Multi-type ancestral process}
Suppose that one has sampled $n\in[\sum_{k\in E}N_k]$ individuals from the current generation $0$. One may order these $n$ individuals (in some arbitrary way) and let $k_1,\ldots,k_n\in E$ denote their types. Looking $r\in\nz_0:=\{0,1,\ldots\}$ generations backward in time, the ancestry of the sample can be captured by defining a random labeled partition ${\cal A}_r={\cal A}_r^{(n,N)}$ of $[n]$ such that (by definition) $i,j\in[n]$ belong to the same $k$-block of ${\cal A}_r$ if and only if the individuals $i$ and $j$ have a common ancestor $r$ generations backward in time and this ancestor has type $k$. Note that ${\cal A}_r$ not only depends on the sample size $n$ and on the subpopulation sizes $N_k$, $k\in E$, but also on the distribution of the offspring numbers $\nu_{k,\ell,i}$, but this dependence is suppressed in our notation for simplicity. The process ${\cal A}:=({\cal A}_r)_{r\in\nz_0}$ is called a \emph{multi-type ancestral process} or, alternatively, a \emph{multi-type backward process} or a \emph{multi-type discrete $n$-coalescent process}. The assumption that offspring sizes in different generations are iid ensures that ${\cal A}$ is a homogeneous Markov chain (HMC) with state space ${\cal P}_{n,E}$ and initial state ${\cal A}_0=\{(\{1\},k_1),\ldots,(\{n\},k_n)\}$. In this case the initial state ${\cal A}_0$ is a deterministic labeled partition of $[n]$ into singletons. Note however, that we could have sampled from generation $0$ according to some different (even random) scheme, for example in such a way that ${\cal A}_0$ is an exchangeable random labeled partition. We will come back to the impact of the distribution of ${\cal A}_0$ at the end of Section \ref{consistency}. Let
   \begin{equation} \label{rtrans1}
      p_{\pi,\pi'}
      \ :=\ \pr({\cal A}_r=\pi'\,|\,{\cal A}_{r-1}=\pi),
      \qquad\pi,\pi'\in{\cal P}_{n,E},r\in\nz,
   \end{equation}
   denote the transition probabilities of ${\cal A}$. Note that $p_{\pi,\pi'}=p_{\pi,\pi'}^{(N)}$ depends on the subpopulation sizes $N=(N_k)_{k\in E}$, but this dependence is often suppressed in our notation. Transitions from $\pi\in{\cal P}_{n,E}$ to $\pi'\in{\cal P}_{n,E}$ are only possible if each block of $\pi'$ is a union of some blocks of $\pi$ (types of the blocks disregarded here). We write $\pi\subseteq\pi'$ in this case. For $x\in\rz$ and $n\in\nz_0$ let $(x)_n:=\prod_{i=0}^{n-1}(x-i)$ denote the descending factorials. Since, in each subpopulation, offspring to parents are randomly assigned (random assignment condition), it follows under (A) that the transition probability (\ref{rtrans1}) can be expressed in terms of the offspring variables $\nu_{k,\ell,s}$, $k,\ell\in E$, $s\in[j_k]$, via
   \begin{equation} \label{rtrans2}
      p_{\pi,\pi'}\ =\
      \frac{\prod_{k\in E}(N_k)_{j_k}}{\prod_{\ell\in E}(N_\ell)_{i_\ell}}
      \me\bigg(
      \prod_{k,\ell\in E}
            \prod_{s=1}^{j_k}(\nu_{k,\ell,s})_{i_{k,\ell,s}}
      \bigg),\qquad\pi,\pi'\in{\cal P}_{n,E},\pi\subseteq\pi',
   \end{equation}
   where $i_\ell$ and $j_k$ are the number of $\ell$-blocks of $\pi$ and $k$-blocks of $\pi'$ respectively and $i_{k,\ell,s}$, $k,\ell\in E$, $s\in[j_k]$, are the group sizes of $\ell$-blocks of $\pi$ merging to the $s$-th $k$-block of $\pi'$. A proof of (\ref{rtrans2}) is provided in Section \ref{proofs}. The structure of the transition matrix $P$ is illustrated exemplary for sample size $n=2$ in Section \ref{structure} in the appendix. Note that $\sum_{k\in E}\sum_{s=1}^{j_k}i_{k,\ell,s}=i_\ell$, $\ell\in E$. In particular, the transition matrix $P:=(p_{\pi,\pi'})_{\pi,\pi'\in{\cal P}_{n,E}}$ has diagonal entries
   \begin{equation}
      p_{\pi,\pi}\ =\ \me\bigg(\prod_{k\in E}\prod_{s=1}^{i_k}\nu_{k,k,s}\bigg),
      \qquad \pi\in{\cal P}_{n,E}.
   \end{equation}
   Consistency properties of the transition probabilities (\ref{rtrans2}) and symmetry properties of the ancestral process are deferred to Section \ref{consistency}.
%
   If $N_{k,\ell}=0$ for all $k\ne\ell$, then (\ref{rtrans2}) reduces to the transition probabilities \cite[Eq.~(9)]{Moehle2024} for multi-type Cannings models, where each offspring has the same type as its parent. For $E=\{1\}$, Eq.~(\ref{rtrans2}) reduces to the well-known formula for the transition probabilities of the discrete coalescent for single-type Cannings models (see, for example, \cite[Eq.~(3)]{MoehleSagitov2001})
   \begin{equation} \label{singletypetrans}
      p_{\pi,\pi'}
      \ =\ \frac{(N)_j}{(N)_i}\me\big((\nu_1)_{i_1}\cdots(\nu_j)_{i_j}\big),
   \end{equation}
   where $N$ $(:=N_1)$ is the total population size, $i$ and $j$ are the number of blocks of $\pi$ and $\pi'$ respectively, $\nu_s:=\nu_{1,1,s}$ for $s\in[N]$ and $i_1,\ldots,i_j$ are the group sizes of merging blocks of $\pi$. Note that $i_1+\cdots+i_j=i$.

   The coalescence probability that two individuals of the same type $\ell\in E$ share a common parent of type $k\in E$ one generation backward in time is
   \begin{equation} \label{coal1}
      \me\bigg(\sum_{i=1}^{N_k}\frac{(\nu_{k,\ell,i})_2}{(N_\ell)_2}\bigg)
      \ =\
      \frac{N_k}{(N_\ell)_2}\me((\nu_{k,\ell,1})_2)
      \ =\ \frac{N_k}{(N_\ell)_2}\me(\nu_{k,\ell,1}^2)-\frac{\me(N_{k,\ell})}{(N_\ell)_2}
      \ =:\ c_{k,\ell}(N_k,N_\ell),
   \end{equation}
   provided that $N_\ell>1$, in agreement with (\ref{rtrans2}) for $i_\ell:=2$, $i_s:=0$ for $s\in E\setminus\{\ell\}$, $j_k:=1$ and $j_s:=0$ for $s\in E\setminus\{k\}$.

   Similarly, for $k,\ell_1,\ell_2\in E$ with $\ell_1\ne\ell_2$, the coalescence probability that two individuals of different types $\ell_1$ and $\ell_2$ respectively share a common parent of type $k$ one generation backward in time is
   \begin{equation} \label{coal2}
      \me\bigg(\sum_{i=1}^{N_k}\frac{\nu_{k,\ell_1,i}}{N_{\ell_1}}
      \frac{\nu_{k,\ell_2,i}}{N_{\ell_2}}\bigg)
      \ =\
      \frac{N_k}{N_{\ell_1}N_{\ell_2}}\me(\nu_{k,\ell_1,1}\nu_{k,\ell_2,1})
      \ =:\ c_{k,\ell_1,\ell_2}(N_k,N_{\ell_1},N_{\ell_2}),
   \end{equation}
   again in agreement with (\ref{rtrans2}) for $i_{\ell_1}:=i_{\ell_2}:=1$,
   $i_s:=0$ for $s\in E\setminus\{\ell_1,\ell_2\}$, $j_k:=1$ and $j_s:=0$ for $s\in E\setminus\{k\}$. We call $c_{k,k}(N_k,N_k)$, $k\in E$, the \emph{diagonal coalescence probabilities} and all the other coalescence probabilities the \emph{off-diagonal coalescence probabilities}. There are $d:=|E|$ diagonal coalescence probabilities and $(d)_2+d(d)_2=d^3-d$ off-diagonal coalescence probabilities, thus, altogether $d^3$ coalescence probabilities.

   For $r\in\nz_0$ and $k\in E$ let $Y_{r,k}$ denote the number of $k$-blocks of ${\cal A}_r$ and define $Y_r:=(Y_{r,k})_{k\in E}$. Using an argument in the spirit of Burke and Rosenblatt \cite{BurkeRosenblatt1958} it is readily seen that $Y:=(Y_r)_{r\in\nz_0}$, called the (multi-type) block counting process of ${\cal A}$, is a HMC with (partially ordered) state space $\nz_0^E$. Let $p_{i,j}:= \pr(Y_r=j\,|\,Y_{r-1}=i)$, $i,j\in\nz_0^E$, $r\in\nz$, denote the transition probabilities of $Y$. From (\ref{rtrans2}) it follows that, for all $i,j\in\nz_0^E$ with $i\ge j$ (componentwise),
   \[
   p_{i,j}
   \ =\ \frac{\prod_{k\in E}\binom{N_k}{j_k}}{\prod_{\ell\in E}\binom{N_\ell}{i_\ell}}
   \sum_{}\me\bigg(
   \prod_{k,\ell\in E}\prod_{s=1}^{j_k}\binom{\nu_{k,\ell,s}}{i_{k,\ell,s}}
   \bigg),
   \]
   where the sum $\sum$ extents over all $i_{k,\ell,s}\in\nz_0$, $k,\ell\in E$, $s\in[j_k]$, with $\sum_{k\in E}\sum_{s=1}^{j_k}i_{k,\ell,s}=i_\ell$ for all $\ell\in E$. It is often useful to study the $\nz_0^E$-valued process $Y$ before considering the full partition-valued process $\Pi$.
\subsection{Consistency and symmetry} \label{consistency}
Single-type Cannings models satisfy a fundamental consistency property (see Eq.~(\ref{singletypeconsis}) below) being crucial for the analysis (of the ancestral structure) of these models. Proposition \ref{consisprop} below shows that multi-type Cannings models satisfy a similar but more involved multi-type consistency property. In order to state the result, it is useful to introduce matrices of the form $T:=(t_{k,\ell})_{k,\ell\in E}$, where each entry $t_{k,\ell}$ of $T$ is a (possibly empty) vector of the form $t_{k,\ell}:=(i_{k,\ell,s})_{s\in[j_k]}$ with $j_k\in\nz_0$ for all $k\in E$ and $i_{k,\ell,s}\in\nz_0$ for all $k,\ell\in E$ and all $s\in[j_k]$ satisfying $i_\ell:=\sum_{k\in E}\sum_{s=1}^{j_k}i_{k,\ell,s}\le N_\ell$ for all $\ell\in E$. If $j_k=0$ then $t_{k,\ell}=()=:\mathbf{0}\in\rz^0$ is the empty vector (neutral and only element of $\rz^0$) for every $\ell\in E$. The symbol $T$ is used since $T$ is a tensor. With this notation, the transition probability (\ref{rtrans2}) is of the form 
\begin{equation} \label{Phi}
   p_{\pi,\pi'}\ =\ \Phi_j(T)\quad\mbox{with}\quad\Phi_j(T)\ :=\ \Phi_j^{(N)}(T)\ :=\
   \frac{\prod_{k\in E}(N_k)_{j_k}}{\prod_{\ell\in E}(N_\ell)_{i_\ell}}
   \me\bigg(
      \prod_{k,\ell\in E}\prod_{s=1}^{j_k}(\nu_{k,\ell,s})_{i_{k,\ell,s}}
   \bigg),
\end{equation}
where $j:=(j_k)_{k\in E}$ and $T$ is the tensor defined above. Note that, if $j=e_k$ is the $k$-th unit vector in $\rz^E$, then
\[
\Phi_{e_k}(T)\ =\ \frac{N_k}{\prod_{\ell\in E}(N_\ell)_{i_{k,\ell,1}}}
\me\bigg(\prod_{\ell\in E}(\nu_{k,\ell,1})_{i_{k,\ell,1}}\bigg),
\qquad k\in E.
\]
In particular, the coalescence probability (\ref{coal1}) has the form $c_{k,\ell}(N_k,N_\ell)=\Phi_{e_k}(T)$, where $T:=(t_{k',\ell'})_{k',\ell'\in E}$ is the tensor with entries $t_{k',\ell'}:=(2)$ if $k'=k$ and $\ell'=\ell$, $t_{k',\ell'}:=(0)=\mathbf{0}\in\rz^1$ if $k'=k$ and $\ell'\ne\ell$ and $t_{k',\ell'}:=()=\mathbf{0}\in\rz^0$ otherwise. Similarly, the coalescence probability (\ref{coal2}) has the form $c_{k,\ell_1,\ell_2}(N_k,N_{\ell_1},N_{\ell_2})=\Phi_{e_k}(T)$, where $T:=(t_{k',\ell'})_{k',\ell'\in E}$ is the tensor with entries $t_{k',\ell'}:=(1)$ if $k'=k$ and $\ell'\in\{\ell_1,\ell_2\}$, $t_{k',\ell'}:=(0)=\mathbf{0}\in\rz^1$ if $k'=k$ and $\ell'\notin\{\ell_1,\ell_2\}$ and $t_{k',\ell'}:=()=\mathbf{0}\in\rz^0$ otherwise.

In the following, for given $j=(j_k)_{k\in E}\in\nz_0^E$, the space of all tensors $T=(t_{k,\ell})_{k,\ell\in E}$ with $t_{k,\ell}=(i_{k,\ell,s})_{s\in[j_k]}$, where the entries $i_{k,\ell,s}\in\nz_0$ satisfy $i_\ell:=\sum_{k\in E}\sum_{s=1}^{j_k}i_{k,\ell,s}\le N_\ell$ for all $\ell\in E$, is denoted by ${\cal T}_j$. Note that ${\cal T}_j={\cal T}_j^{(N)}$ depends on the subpopulation sizes $N:=(N_k)_{k\in E}$ and that $\Phi_j(T)$ is well defined for all $T\in{\cal T}_j$. Thus, $\Phi_j$ is a function from ${\cal T}_j$ to $[0,1]$. The following proposition provides full information on the consistency property of multi-type Cannings models. Despite the fact that its proof, provided in Section \ref{proofs}, is rather short, this consistency property turns out to be crucial for essentially all what follows.
\begin{proposition}[Multi-type consistency] \label{consisprop}
   Under (A), the functions $\Phi_j:{\cal T}_j\to[0,1]$, $j:=(j_k)_{k\in E}\in\nz_0^E$, are consistent in the following sense. For all $j=(j_k)_{k\in E}\in\nz_0^E$ and all tensors $T\in{\cal T}_j$, the equality
   \begin{equation} \label{consis}
      \Phi_j(T)\ =\ \sum_{k\in E}\Phi_{j+e_k}\big(T(k,\ell)\big)
      +\sum_{k\in E}\sum_{s=1}^{j_k}\Phi_j\big(T(k,\ell,s)\big)
   \end{equation}
   holds for each $\ell\in E$ with $i_\ell:=\sum_{k\in E}\sum_{s\in[j_k]}i_{k,\ell,s}<N_\ell$, where $e_k$ denotes the \mbox{$k$-th} unit vector in $\rz^E$, the tensor $T(k,\ell)$ is obtained from $T$ by replacing the (possibly empty) vector $t_{k,\ell}=(i_{k,\ell,1},\ldots,i_{k,\ell,j_k})$ by $(i_{k,\ell,1},\ldots,i_{k,\ell,j_k},1)$ and the (possibly empty) vector $t_{k,\ell'}=(i_{k,\ell',1},\ldots,i_{k,\ell',j_k})$ by $(i_{k,\ell',1},\ldots,i_{k,\ell',j_k},0)$ for all $\ell'\ne\ell$, and the tensor $T(k,\ell,s)$ is obtained from $T$ by replacing the single entry $i_{k,\ell,s}$ by $i_{k,\ell,s}+1$.
\end{proposition}
\begin{remark}
   In particular, the right-hand side of (\ref{consis}) takes the same value for each $\ell\in E$ with $i_\ell<N_\ell$, which is a-priori not obvious. For $j=(0)_{k\in E}$ (null vector), Eq.~(\ref{consis}) reduces to the normalizing condition
   \begin{equation} \label{consis0}
      1\ =\ \sum_{k\in E}\Phi_{e_k}\big(T(k,\ell)\big),\qquad \ell\in E,
   \end{equation}
   where $T(k,\ell)$ is the tensor with entries $t_{k,\ell}:=(1)$, $t_{k,\ell'}:=(0)$ for $\ell'\ne\ell$ and $t_{k',\ell'}:=()$ (empty vector) otherwise. Eq.~(\ref{consis0}) is easily seen as follows. For all $k,\ell\in E$ we have $\Phi_{e_k}(T(k,\ell))=(N_k/N_\ell)\me(\nu_{k,\ell,1})=\me(N_{k,\ell})/N_\ell$. Summing over all $k\in E$ and taking (\ref{size1}) into account yields (\ref{consis0}).

   For $E=\{1\}$ and $j:=j_1\in\nz$, Eq.~(\ref{consis}) reduces to the consistency equation for single-type Cannings models (see, for example, \cite[Eq.~(3) and (4)]{Moehle2002})
   \begin{equation} \label{singletypeconsis}
      \Phi_j(i_1,\ldots,i_j)\ =\ \Phi_{j+1}(i_1,\ldots,i_j,1)+
      \sum_{s=1}^j \Phi_j(i_1,\ldots,i_{s-1},i_s+1,i_{s+1},\ldots,i_j),
   \end{equation}
   $i_1,\ldots,i_j\in\nz$ with $i_1+\cdots+i_j<N$ ($:=N_1$). The multi-type consistency property provided in Proposition \ref{consisprop} is more involved than its single-type counterpart (\ref{singletypeconsis}) in the sense that it cannot be derived from (\ref{singletypeconsis}) alone. For $E=\{1\}$, the normalizing condition (\ref{consis0}) reduces to $\Phi_1(1)=1$. Note that the same consistency relation (\ref{singletypeconsis}) holds for all $i_1,\ldots,i_j\in\nz$ if $\Phi_j$ is the exchangeable partition probability function (EPPF) of an infinite exchangeable random partition (see, for example, Pitman \cite[Eq.~(2.9)]{Pitman2006}).
\end{remark}
As in the single-type case, multi-type consistency has fundamental consequences, two of the probably most important of them provided in Corollary \ref{monotonecorollary} and Corollary \ref{naturalcoupling} below.

Let $j,j'\in\nz_0^k$ with $j\le j'$ (componentwise) and $T\in{\cal T}_j$, $T'\in{\cal T}_{j'}$ be two tensors. Recall that $T=(t_{k,\ell})_{k,\ell\in E}$ with $t_{k,\ell}=(i_{k,\ell,s})_{s\in[j_k]}$ and, similarly, $T'=(t'_{k,\ell})_{k,\ell\in E}$ with $t'_{k,\ell}=(i'_{k,\ell,s})_{s\in[j_k']}$. We say that $T\le T'$ if $i_{k,\ell,s}\le i'_{k,\ell,s}$ for all $k,\ell\in E$ and all $s\in[j_k]$. The following monotonicity property is a direct consequence of the consistency. Again, its proof is provided in Section \ref{proofs}.
\begin{corollary}[Monotonicity] \label{monotonecorollary}
   Under (A), the functions $\Phi_j:{\cal T}_j\to [0,1]$, $j\in\nz_0^E$, are monotone in the sense that
   \begin{equation} \label{monotone}
      \Phi_{j'}(T')\ \le\ \Phi_j(T)
   \end{equation}
   for all $j,j'\in\nz_0^E$ with $j\le j'$ and all tensors $T\in{\cal T}_j$ and $T'\in{\cal T}_{j'}$ with $T\le T'$.
\end{corollary}
For $m,n\in\nz$ with $m\le n$ let $\varrho_{n,m}:{\cal P}_{n,E}\to{\cal P}_{m,E}$ denote the natural restriction from ${\cal P}_{n,E}$ to ${\cal P}_{m,E}$ defined via
\begin{equation} \label{restriction}
   \varrho_{n,m}(\pi)
   \ :=\ \{(B_i\cap[m],k_i):1\le i\le j,B_i\cap[m]\ne\emptyset\}
\end{equation}
for all $\pi=\{(B_1,k_1),\ldots,(B_j,k_j)\}\in{\cal P}_{n,E}$. The following corollary shows that multi-type Cannings models satisfy the natural coupling property. For general information on the natural coupling property of the Kingman $n$-coalescent we refer the reader to Section 7 of \cite{Kingman1982a}.
\begin{corollary}[Natural coupling] \label{naturalcoupling}
   Under (A), for all $m,n\in\nz$ with $m\le n$, the restricted process $(\varrho_{n,m}\circ{\cal A}_r^{(n)})_{r\in\nz_0}$ has the same distribution as the multi-type ancestral process $({\cal A}_r^{(m)})_{r\in\nz_0}$.
\end{corollary}
We now turn to symmetry and exchangeability properties of multi-type Cannings models. For $n\in\nz$ let $S_n$ be the set of permutations of $[n]$. Each permutation $\sigma\in S_n$ acts on $\rz^n$ via $\sigma x:=(x_{\sigma (1)},\ldots,x_{\sigma(n)})$ for all $x=(x_1,\ldots,x_n)\in\rz^n$. This definition is naturally extended to $n=0$ by assuming that $S_0$ contains only one particular permutation $\sigma_0$ acting on the empty vector $()\in\rz^0$ via $\sigma_0():=()$.

Let $j=(j_k)_{k\in E}\in\nz_0^E$. The definition of the function $\Phi_j$ (see (\ref{Phi})) and (A) imply that $\Phi_j$ is symmetric in the following sense. For all tensors $T=(t_{k,\ell})_{k,\ell\in E}\in{\cal T}_j$ and all permutations $\sigma_{k,\ell}\in S_{j_k}$, $k,\ell\in E$,
\begin{equation} \label{phisym}
   \Phi_j(T)\ =\ \Phi_j(\sigma(T)),
\end{equation}
where $\sigma:=(\sigma_{k,\ell})_{k,\ell\in E}$ and the tensor $\sigma(T)\in{\cal T}_j$ is defined via $\sigma(T):=(\sigma_{k,\ell}t_{k,\ell})_{k,\ell\in E}$.

The distribution of the ancestral process $({\cal A}_r)_{r\in\nz_0}$ clearly does not only depend on the transition probabilities (\ref{Phi}) but also on the distribution of ${\cal A}_0$. Lemma \ref{exchangelemma} below shows that the ancestral process is exchangeable, if one assumes that the individuals from generation $0$ are sampled in such a way that ${\cal A}_0$ is an exchangeable random labeled partition. For example, one may sample all $n$ individuals from a single subpopulation or one may sample the $n$ individuals in an exchangeable manner such that their random types are exchangeable $E$-valued random variables. The proof of Lemma \ref{exchangelemma} is provided on Section \ref{proofs}.
\begin{lemma}[Exchangeability] \label{exchangelemma}
   If ${\cal A}_0$ is exchangeable then ${\cal A}_r$ is exchangeable for all $r\in\nz_0$.
\end{lemma}
\begin{remark}
   If ${\cal A}_0$ is not exchangeable, then all the ${\cal A}_r$, $r\in\nz$, are in general not exchangeable. Still, the ancestral process shares some `reduced' symmetry properties. To explain this it turns out to be useful to introduce certain subgroups of the permutation group $S_n$ as follows.

   Let $k_1,\ldots,k_n\in E$. For $\ell\in E$ define $A_\ell:=\{i\in[n]:k_i=\ell\}$. Note that $A_\ell=A_\ell^{(k_1,\ldots,k_n)}$ depends on $k_1,\ldots,k_n$ and that $\sum_{\ell\in E}|A_\ell| =n$. For $\ell\in E$ let $\sigma_\ell$ be a permutation of $A_\ell$ and for $i\in A_\ell$ define $\sigma(i):=\sigma_\ell(i)$. Use the symbol $S_{k_1,\ldots,k_n}$ for the set of all such permutations $\sigma$. Note that $S_{k_1,\ldots,k_n}$ is a subgroup of $S_n$.

   Now, let $\pi\in{\cal P}_{n,E}$ and let $k_i=k_i(\pi)$ denote the type of individual $i\in[n]$. Then, for every generation $r\in\nz$, the random labeled partition ${\cal A}_r$ is partially exchangeable with respect to $S_{k_1,\ldots,k_n}$ on the event $\{{\cal A}_{r-1}=\pi\}$ in the sense of Definition \ref{partialexchangeability}, i.e.,
   \[
   \pr(\sigma({\cal A}_r)=\pi',{\cal A}_{r-1}=\pi)\ =\ \pr({\cal A}_r=\pi',{\cal A}_{r-1}=\pi)
  \]
   for all $\sigma\in S_{k_1,\ldots,k_n}$ and $\pi'\in{\cal P}_{n,E}$.

   For the single-type situation $|E|=1$, all the $k_1,\ldots,k_n$ are equal and, hence, the subgroups $S_{k_1,\ldots,k_n}$ coincide with the full permutation group $S_n$. Thus, in this case, the equality $\pr(\sigma({\cal A}_r)=\pi',{\cal A}_{r-1}=\pi)= \pr({\cal A}_r=\pi',{\cal A}_{r-1}=\pi)$ holds for all $\sigma\in S_n$ and $\pi,\pi'\in{\cal P}_n$. Summing over all $\pi$ shows that $\sigma({\cal A}_r)\stackrel{d}{=}{\cal A}_r$ for all $\sigma\in S_n$. Thus, ${\cal A}_r$ is exchangeable. Similar issues concerning $d$-type exchangeability for infinite labeled partitions and $d$-type coalescents are addressed in \cite[Section 3.1]{JohnstonKyprianouRogers2023}.
\end{remark}
\subsection{Limiting results} \label{limits}
It is natural to investigate the behaviour of the ancestral process ${\cal A}^{(n,N)}=({\cal A}_r^{(n,N)})_{r\in\nz_0}$ for large subpopulation sizes $N_k$, $k\in E$, i.e., when the minimal subpopulation size
\begin{equation} \label{Nmin}
   N_{\rm min}\ :=\ \min_{k\in E}N_k
\end{equation}
tends to infinity. In the following, $P_N=(p_{\pi,\pi'}^{(N)})_{\pi,\pi'\in{\cal P}_{n,E}}$ denotes the transition matrix of the ancestral process $({\cal A}_r^{(n,N)})_{r\in\nz_0}$ and $I$ the identity matrix (of the same size as $P_N$).
\begin{lemma} \label{identitylemma}
   Assume that $n\ge 2$. Then the convergence $P_N\to I$ as $N_{\rm min}\to\infty$ holds if and only if $\me(\nu_{k,k,1})\to 1$ and $\me(\nu_{k,k,1}\nu_{k,k,2})\to 1$ as $N_{\rm min}\to\infty$ for all $k\in E$. In this case, $(N_k/N_\ell)^2\me(\nu_{k,\ell,1}\nu_{k,\ell,2})\to\delta_{k,\ell}$ (Kronecker symbol) as $N_{\rm min}\to\infty$ for all $k,\ell\in E$.
\end{lemma}
Let us first focus on a situation, where it will turn out that the ancestral process is in the domain of attraction of a discrete-time limiting multi-type process. Let $j=(j_k)_{k\in E}\in\nz_0^E$. A tensor $T=(t_{k,\ell})_{k,\ell\in E}\in{\cal T}_j$ is called a \emph{diagonal tensor}, if $t_{k,\ell}=\mathbf{0}$ ($\in\rz^{j_k}$) for all $k,\ell\in E$ with $k\ne\ell$. The particular diagonal tensor $(t_{k,\ell})_{k,\ell\in E}\in{\cal T}_j$ with diagonal entries $t_{k,k}:=(1,\ldots,1)\in\rz^{j_k}$ for all $k\in E$ is denoted by $\mathbf{1}_j$.
\begin{theorem} (Convergence of the ancestral process, discrete-time limit) \label{main1}\\
   Assume that
      for all $j=(j_k)_{k\in E}\in\nz_0^E$ and all tensors $T=(t_{k,\ell})_{k,\ell\in E}\in{\cal T}_j\setminus\{\mathbf{1}_j\}$, 
          the convergence $\Phi_j^{(N)}(T)\to\phi_j(T)$ as $N_{\rm min}\to\infty$ holds for some constant $\phi_j(T)\in[0,1]$.
   Let $n\in\nz$. If ${\cal A}_0^{(n,N)}\to\Pi_0^{(n)}$ in distribution as $N_{\rm min}\to\infty$ for some random labeled partition $\Pi_0^{(n)}$ of $[n]$, then the multi-type ancestral process $({\cal A}_r^{(n,N)})_{r\in\nz_0}$ converges in $D_{{\cal P}_{n,E}}(\nz_0)$ as $N_{\rm min}\to\infty$ to a discrete-time limiting Markov chain $\Pi^{(n)}=(\Pi_r^{(n)})_{r\in\nz_0}$ with state space ${\cal P}_{n,E}$ and transition matrix $A:=(a_{\pi,\pi'})_{\pi,\pi'\in{\cal P}_{n,E}}$ having non-diagonal entries $a_{\pi,\pi'}:=\phi_j(T)$ if $\pi\subset\pi'$ and $a_{\pi,\pi'}:=0$ otherwise, and diagonal entries $a_{\pi,\pi}:=1-\sum_{\pi'\ne\pi'} a_{\pi,\pi'}$.
\end{theorem}
For many multi-type Cannings models (see Lemma \ref{identitylemma}), the transition matrix $P_N$ of the ancestral process satisfies $P_N\to I$ as $N_{\rm min}\to\infty$. This holds for example under a sort of weak mutation assumption for the multi-type Wright--Fisher model discussed in Section \ref{mwfm}. In this situation, the limits $\phi_j(T)$ in 
Theorem \ref{main1} are all equal to zero and, hence, Theorem \ref{main1} is not useful, since the limiting Markov chain $\Pi^{(n)}$ satisfies $\Pi_r^{(n)}=\Pi_0^{(n)}$ almost surely for all $r\in\nz_0$. Thus, other assumptions are required in this situation to obtain convergence of the ancestral process to a non-degenerate limiting process. The following result covers such cases and it turns out that the limiting process is a continuous-time multi-type process.
\begin{theorem} (Convergence of the ancestral process, continuous-time limit) \label{main2}\\
   Suppose that, for every $N=(N_k)_{k\in E}$, there exists $c_N>0$ such that $c_N\to 0$ as $N_{\rm min}\to\infty$ and such that, 
       for all $j=(j_k)_{k\in E}\in\nz_0^E$ and all tensors $T=(t_{k,\ell})_{k,\ell\in E}\in{\cal T}_j\setminus\{\mathbf{1}_j\}$, 
       the convergence $\Phi_j^{(N)}(T)/c_N\to\phi_j(T)$ as $N_{\rm min}\to\infty$ holds for some real constant $\phi_j(T)$.
   Let $n\in\nz$. If ${\cal A}_0^{(n,N)}\to\Pi_0^{(n)}$ in distribution as $N_{\rm min}\to\infty$ for some random labeled partition $\Pi_0^{(n)}$ of $[n]$, then the time-scaled multi-type ancestral process $({\cal A}_{\lfloor t/c_N\rfloor}^{(n,N)})_{t\ge 0}$ converges in $D_{{\cal P}_{n,E}}([0,\infty))$ as $N_{\rm min}\to\infty$ to a continuous-time limiting Markov process $\Pi^{(n)}=(\Pi_t^{(n)})_{t\ge 0}$ with state space ${\cal P}_{n,E}$ and infinitesimal generator $Q:=(q_{\pi,\pi'})_{\pi,\pi'\in{\cal P}_{n,E}}$ having non-diagonal entries $q_{\pi,\pi'}:=\phi_j(T)$ if $\pi\subset\pi'$ and $q_{\pi,\pi'}:=0$ otherwise, and diagonal entries $q_{\pi,\pi}:=-\sum_{\pi'\ne\pi}q_{\pi,\pi'}$, $\pi\in{\cal P}_{n,E}$.
\end{theorem}
\begin{remark}
   The standard choice for the time-scaling is the supremum over all coalescence probabilities
   \begin{equation} \label{scaling}
      c_N\ :=\ \sup\{
        c_{k,\ell}(N_k,N_\ell),c_{k,\ell_1,\ell_2}(N_k,N_{\ell_1},N_{\ell_2})
        :k,\ell,\ell_1,\ell_2\in E,\ell_1\ne\ell_2
      \},
   \end{equation}
   where $c_{k,\ell}(N_k,N_\ell)$ and $c_{k,\ell_1,\ell_2}(N_k,N_{\ell_1},N_{\ell_2})$ are the coalescence probabilities defined via (\ref{coal1}) and (\ref{coal2}), in agreement with the standard choice (see \cite{MoehleSagitov2001}) $c_N=\Phi_1(2)$ for the $1$-type case. Intuitively, if time is measured in units of $\lfloor 1/c_N\rfloor$ generations, a possible limiting process of the time-scaled ancestral process as $N_{\rm min}\to\infty$ should capture the genealogy for all parts of the ancestral process emerging from maximal coalescence probabilities, i.e., for all parts having the fastest evolution backward in time. This intuition is made rigorous by Theorem \ref{main2} above. If $c_N\to c$ as $N_{\rm min}\to\infty$ for some constant $c>0$, then we are back in the situation of Theorem \ref{main1}, where $c_N$ is not needed for the formulation of the convergence result. In Section \ref{mwfm}, Theorem \ref{main2} is applied with $c_N:=1/N_{\rm min}$ to the multi-type Wright--Fisher model leading in the limit to the multi-type Kingman $n$-coalescent.
\end{remark}

\subsection{Multi-type exchangeable coalescents} \label{coal}

In the following we restrict our attention to the continuous-time setting in the spirit of Theorem \ref{main2}. Similar results hold for the discrete-time setting considered in Theorem \ref{main1}, but are omitted here for the sake of shortness and simplicity.

For $j=(j_k)_{k\in E}\in\nz_0^E$ let ${\cal T}_j$ denote the set of all tensors $T=(t_{k,\ell})_{k,\ell\in E}$ with vector entries $t_{k,\ell}\in\nz_0^{j_k}$ for all $k,\ell\in E$. As in Theorem \ref{main2} it is from now on assumed that for every $N=(N_k)_{k\in E}$ there exists a constant $c_N>0$ such that the limits
\begin{equation} \label{philimits}
   \phi_j(T)\ :=\ \lim_{N_{\rm min}\to\infty}\frac{\Phi_j^{(N)}(T)-\delta_{T,\mathbf{1}_j}}{c_N}
\end{equation}
exist for all $j\in\nz_0^E$ and all tensors $T\in{\cal T}_j$. Note that $\phi_j(\mathbf{1}_j)\le 0$ for all $j\in\nz_0^E$. In the following it is shown that the consistency equation (\ref{consis0}) of the functions $\Phi_j^{(N)}$ carries over to the limits (\ref{philimits}). More precisely, for all $j=(j_k)_{k\in E}\in\nz_0^E$ and all $T=(t_{k,\ell})_{k,\ell\in E}\in{\cal T}_j$, the equality
\begin{equation} \label{phiconsis}
   \phi_j(T)\ =\ \sum_{k\in E}\phi_{j+e_k}(T(k,\ell)) + \sum_{k\in E}\sum_{s=1}^{j_k}\phi_j(T(k,\ell,s))
\end{equation}
holds for all $\ell\in E$, where $e_k$ denotes the $k$-th unit vector in $\rz^E$, the tensor $T(k,\ell)$ is obtained from $T$ by replacing the vector $t_{k,\ell}=(i_{k,\ell,1},\ldots,i_{k,\ell,j_k})$ by $(i_{k,\ell,1},\ldots,i_{k,\ell,j_k},1)$ and the (possibly empty) vector $t_{k,\ell'}=(i_{k,\ell',1},\ldots,i_{k,\ell',j_k})$ by $(i_{k,\ell',1},\ldots,i_{k,\ell',j_k},0)$ for all $\ell'\ne\ell$, and the tensor $T(k,\ell,s)$ is obtained from $T$ by replacing the single entry $i_{k,\ell,s}$ by $i_{k,\ell,s}+1$. For $T\in{\cal T}_j\setminus\{\mathbf{1}_j\}$, Eq.~(\ref{phiconsis}) follows by dividing Eq.~(\ref{consis}) by $c_N$ and taking the limit $N_{\rm min}\to\infty$. For the particular tensor $T=\mathbf{1}_j$, we conclude from (\ref{consis}) that, for all $\ell\in E$,
\begin{eqnarray*}
   \frac{\Phi_j^{(N)}(\mathbf{1}_j)-1}{c_N}
   & = & \sum_{k\in E}\frac{\Phi_{j+e_k}^{(N)}(\mathbf{1}_j(k,\ell))}{c_N}
         + \sum_{k\in E}\sum_{s=1}^{j_k}
         \frac{\Phi_j^{(N)}(\mathbf{1}_j(k,\ell,s))}{c_N}-\frac{1}{c_N}\\
   & = & \frac{\Phi_{j+e_\ell}^{(N)}(\mathbf{1}_{j+e_\ell})-1}{c_N}
+ \sum_{k\ne\ell}\frac{\Phi_{j+e_k}^{(N)}(\mathbf{1}_j(k,\ell))}{c_N}
+ \sum_{k\in E}\sum_{s=1}^{j_k}\frac{\Phi_j^{(N)}(\mathbf{1}_j(k,\ell,s))}{c_N}.
\end{eqnarray*}
Taking the limit $N_{\rm min}\to\infty$ yields
\begin{eqnarray*}
   \phi_j(\mathbf{1}_j)
   & = & \phi_{j+e_\ell}(1_{j
         +e_\ell})+\sum_{k\ne\ell}\phi_{j+e_k}(\mathbf{1}_j(k,\ell))
         +\sum_{k\in E}\sum_{s=1}^{j_k}\phi_j(\mathbf{1}_j(k,\ell,s))\\
   & = & \sum_{k\in E}\phi_{j+e_k}(\mathbf{1}_j(k,\ell))+\sum_{k\in E}\sum_{s=1}^{j_k}\phi_j(\mathbf{1}_j(k,\ell,s)),
\end{eqnarray*}
which shows that (\ref{phiconsis}) holds as well for $T=\mathbf{1}_j$.

From Corollary \ref{monotonecorollary} it follows that the functions $\phi_j$, $j=(j_k)_{k\in E}\in\nz_0^E$, are monotone in the sense that
\begin{equation} \label{phimon}
   0\ \le\ \phi_{j'}(T')\ \le\ \phi_j(T)
\end{equation}
for all $j,j'\in\nz_0^E$ with $j\le j'$ and all tensors $T\in{\cal T}_j\setminus\{\mathbf{1}_j\}$ and $T'\in{\cal T}_{j'}\setminus\{\mathbf{1}_{j'}\}$ with $T\le T'$. Furthermore, $\phi_{j'}(\mathbf{1}_{j'})\le\phi_j(\mathbf{1}_j)\le 0$ for all $j,j'\in\nz_0^E$ with $j\le j'$.

Following the proof of Corollary \ref{naturalcoupling}, but with conditional probabilities replaced by infinitesimal rates (see also Burke and Rosenblatt \cite{BurkeRosenblatt1958}), it follows by exploiting the consistency equations (\ref{phiconsis}) that the limiting processes $\Pi^{(n)}$, $n\in\nz$, satisfy the natural coupling property, i.e., for all $m,n\in\nz$ with $m\le n$, the process $(\varrho_{n,m}\circ \Pi_t^{(n)})_{t\ge 0}$ has the same distribution as $\Pi^{(m)}$. The Daniell--Kolmogorov extension theorem ensures that there exists a process $\Pi$ taking values in the space ${\cal P}_{\infty,E}$ of labeled partitions of $\nz$ such that, for each $n\in\nz$, the restricted process $(\varrho_n\circ\Pi_t)_{t\ge 0}$ has the same distribution as $\Pi^{(n)}$. The process $\Pi$ is called a \emph{multi-type coalescent}. Its distribution is determined by the rate functions $\phi_j$, $j\in\nz_0^E$. We have therefore verified the existence of a general class of multi-type coalescent processes allowing for simultaneous multiple collisions of ancestral lineages. Note that if ${\cal A}_0$ is exchangeable, then, by Lemma \ref{exchangelemma}, the ancestral process $({\cal A}_r^{(n,N)})_{r\in\nz_0}$ is exchangeable and this exchangeability carries over to the limiting multi-type coalescent $\Pi$. As we shall see in the following section, the structure and analysis of these multi-type exchangeable coalescent processes is more involved as it seems to be at a first glance and leads to further open questions.
\subsection{Integral representations} \label{intrep}
It is well-known (Pitman \cite{Pitman1999}, Sagitov \cite{Sagitov1999}) that (single-type) coalescents with multiple collisions can be characterized by a finite measure $\Lambda$ on the unit interval ($\Lambda$-coalescent). Similarly, single-type exchangeable coalescents (allowing for simultaneous multiple mergers of ancestral lineages) can be either characterized (see, for example, M. and Sagitov \cite{MoehleSagitov2001}) by a certain sequence $(Q_j)_{j\in\nz}$ of finite measures $Q_j$ on the finite simplex
\begin{equation} \label{finitesimplex}
   \Delta_j\ :=\ \{(x_1,\ldots,x_j)\in[0,1]^j:\sum_{i=1}^jx_i\le 1\},
\end{equation}
$j\in\nz$, or (Schweinsberg \cite{Schweinsberg2000}; see also Sagitov \cite{Sagitov2003}) by a single finite measure $\Xi$ on the infinite simplex
\begin{equation} \label{simplex}
   \Delta\ :=\ \{x=(x_i)_{i\in\nz}:x_1\ge x_2\ge\cdots\ge 0,\sum_{i\in\nz} x_i\le 1\}
\end{equation}
($\Xi$-coalescent). Note that $\Delta$ is a compact Hausdorff Polish space. In common for these characterizations is the fact that the infinitesimal rates of the coalescent are expressed in terms of certain integrals over the corresponding measure $\Lambda$, the sequence of measures $Q_j$, $j\in\nz$, or the measure $\Xi$, respectively. In this section we aim for similar characterizations for multi-type exchangeable coalescents.

In the following we focus on the particular subclass of multi-type exchangeable coalescents having the additional property that the support of the rate functions $\phi_j$ is concentrated on diagonal tensors $T\in{\cal T}_j$ as already defined at the beginning of Section \ref{limits}. Having a sort of asymptotically weak mutation regime in mind, this is a natural condition.
In this case the general consistency equation (\ref{phiconsis}) considerably simplifies to
\begin{equation} \label{consisdiag}
   \phi_j(T)\ =\ \phi_{j+e_\ell}(T(\ell,\ell)) + \sum_{s=1}^{j_\ell}\phi_j(T(\ell,\ell,s)),\qquad\ell\in E.
\end{equation}

Compared to the rather involved general equations (\ref{phiconsis}), the advantage of (\ref{consisdiag}) is the fact that the values $\phi_j(T)$ of the family of functions $\phi_j$ with diagonal tensors $T$ having only entries greater than or equal to $2$ fully determine the functions $\phi_j$. This crucial fact is readily verified using (\ref{consisdiag}) by induction on the number of ones in the tensor $T$, and paves the way to obtain integral representations in the spirit of M. and Sagitov \cite{MoehleSagitov2001}, Schweinsberg \cite{Schweinsberg2000} and Sagitov \cite{Sagitov2003} for rate functions $\phi_j$, $j\in\nz_0^E$, with support concentrated on diagonal tensors. 

Integral representations for rate functions $\phi_j$, $j\in\nz_0^E$, satisfying the general consistency equations (\ref{phiconsis}), however, remain unknown. The general consistency equations (\ref{phiconsis}) also motivate to study a new class of multi-type exchangeable partition probability functions (M-EPPF) as defined in Section \ref{MEPPF} in the appendix. These stimulating open problems are left for future work.

In the remaining part of this section it is always assumed that the support of each function $\phi_j$ is concentrated on diagonal tensors. We now turn to the desired integral representations.

For $j=(j_k)_{k\in E}\in\nz_0^E$ let $\Delta_j$ denote the simplex of all $x:=(x_{k,s})_{k\in E,s\in[j_k]}$ satisfying $x_{k,s}\ge 0$ for all $k\in E$ and all $s\in[j_k]$ and $\sum_{k\in E}\sum_{s=1}^{j_k}x_{k,s}\le 1$. For $|E|=1$, this definition of $\Delta_j$ is in agreement with (\ref{finitesimplex}).
The following integral representation is the natural extension of \cite[Lemma 3.1]{MoehleSagitov2001}. 
\begin{proposition} \label{int1prop}
   If the limits $\phi_j(T)$ exist for all $j\in\nz_0^E$ and all diagonal tensors $T=(t_{k,\ell})_{k,\ell\in E}\in{\cal T}_j$ with diagonal entries $t_{k,k}\in\{2,3,\ldots\}^{j_k}$ for all $k,\ell\in E$, then there exists for every $j\in\nz_0^E$ a finite measure $Q_j$ on $\Delta_j$ uniquely determined by its moments
   \begin{equation} \label{int1}
      \int_{\Delta_j} \prod_{k\in E}\prod_{s=1}^{j_k}
      x_{k,s}^{i_{k,s}-2}
      \,Q_j\big({\rm d}((x_{k,s})_{k\in E,s\in[j_k]})\big)
      \ =\ \phi_j(T),
   \end{equation}
   $i_{k,s}\in\{2,3,\ldots\}$, $k\in E$, $s\in[j_k]$, where $T:=(t_{k,\ell})_{k,\ell\in E}$ is the diagonal tensor with diagonal entries $t_{k,k}:=(i_{k,s})_{s\in[j_k]}$, $k\in E$. Moreover, the total masses of the family of measures $Q_j$, $j\in\nz_0^E$, satisfy $Q_j(\Delta_j)\ge Q_{j'}(\Delta_{j'})$ for all $j,j'\in\nz_0^E$ with $j\le j'$.
\end{proposition}
\begin{remark}
   If all the measures $Q_j$, $j=(j_k)_{k\in E}\in\nz_0^E$ with $\sum_{k\in E}j_k>1$, are equal to the zero measure, then the corresponding coalescent $\Pi$ is a multi-type $\Lambda$-coalescent characterized by the family $\Lambda:=(\Lambda_k)_{k\in E}$ of measures $\Lambda_k$ defined via $\Lambda_k:=Q_{e_k}$ for all $k\in E$. The measure $\Lambda_k$ characterizes the rates of multiple collisions of $k$-blocks to a single $k$-block. In particular, simultaneous multiple collisions of ancestral lineages are impossible and as well multiple collisions of $k$-blocks to an $\ell$-block with $\ell\ne k$ are impossible for this coalescent.
\end{remark}
The following theorem extends the `only if' part of Schweinsberg \cite[Theorem 2]{Schweinsberg2000}. It characterizes the rates of an exchangeable multi-type coalescent $\Pi$ in terms of a sequence $(a_k)_{k\in E}\in[0,\infty)^E$ and a single finite measure $\Xi$ on $(\Delta\setminus\{0\})\times E^\nz$. Because of its length the proof of Theorem \ref{main3} is provided in the separate Section \ref{proofmain3}. For $x=(x_i)_{i\in\nz}\in\Delta$ define $(x,x):=\sum_{i\in\nz}x_i^2$. 
\begin{theorem} \label{main3}
   Suppose that the rate functions $\phi_j$ of the multi-type coalescent $\Pi$ are concentrated on diagonal tensors. Then there exists a sequence $(a_k)_{k\in E}\in[0,\infty)^E$ and a finite measure $\Xi$ on $(\Delta\setminus\{0\})\times E^\nz$ such that for all $j=(j_k)_{k\in E}\in\nz_0^E$ and all diagonal tensors $T=(t_{k,\ell})_{k,\ell\in E}$ with diagonal entries $t_{k,k}:=(i_{k,s})_{s\in[j_k]}\in\{2,3,\ldots\}^{j_k}$, the representation
   \begin{equation} \label{xichar}
      \phi_j(T)\ =\ \sum_{k\in E}a_k1_{\{j=\mathbf{e}_k,i_{k,1}=2\}}
      +\int_{(\Delta\setminus\{\mathbf{0}\})\times E^\nz}
      \sum_{m_{k,s}} \prod_{k\in E} \prod_{s=1}^{j_k}
      x_{m_{k,s}}^{i_{k,s}}\,1_{\{k\}}(y_{m_{k,s}})
      \,\frac{\Xi({\rm d}(x,y))}{(x,x)}
   \end{equation}
   holds, where 
   the sum $\sum_{m_{k,s}}$ extents over all pairwise distinct $m_{k,s}\in\nz$, $k\in E$, $s\in[j_k]$.
\end{theorem}
\begin{remark}
   For the single-type case, (\ref{xichar}) reduces to the well-known representation of the rates in terms of a finite measure $\Xi$ on $\Delta$ via
   \[
   \phi_j(i_1,\ldots,i_j)\ =\ a1_{\{j=1,i_1=2\}}+\int_{\Delta\setminus\{\mathbf{0}\}} \sum_{\substack{m_1,\ldots,m_j\in\nz\\ \rm all\;distinct}}
   \prod_{s=1}^j x_{m_s}^{i_s}\,\frac{\Xi({\rm d}x)}{(x,x)}
   \]
   for all $j\in\nz$ and $i_1,\ldots,i_j\in\{2,3,\ldots\}$, where $a:=\Xi(\{\mathbf{0}\})$.
\end{remark}
\begin{remark}
   Although rather different in detail, the representation (\ref{xichar}) shares some structural similarities with the L\'evy-Khintchine formula for infinitely divisible distributions on multi-dimensional spaces. We refer the reader exemplary to Applebaum \cite{Applebaum2009} for some background on L\'evy processes and infinite divisibility. The first part $\sum_{k\in E}a_k1_{\{j=e_k,i_{k,1}=2\}}$ in (\ref{xichar}), called the Kingman part, characterizing the binary merging rates of the multi-type exchangeable coalescent. This part can in some sense be viewed as a `normal' or `Brownian' part reflecting the generally accepted statement that the Kingman coalescent is the `normal distribution' of ancestral population genetics. The second (integral) part in (\ref{xichar}) characterises the simultaneous multiple merger rates (resulting from large offspring sizes). This part can be viewed as a sort of jump L\'evy part.
\end{remark}

\subsection{Examples} \label{examples}
In the following two examples the numbers $N_{k,\ell}$, $k,\ell\in E$, are (mainly) deterministic. For simplicity it is also assumed that the type space $E$ is finite, i.e., $d:=|E|\in\nz$.
\subsubsection{Multi-type Wright--Fisher model} \label{mwfm}
Let $N_{k,\ell}\in\nz_0$, $k,\ell\in E$, be given non-negative integer constants with (see (\ref{size1})) $N_\ell:=\sum_{k\in E}N_{k,\ell}\in\nz$ for all $\ell\in E$. Assume that for every $k,\ell\in E$ the vector $\nu_{k,\ell}:=(\nu_{k,\ell,i})_{i\in[N_k]}$ has
a symmetric multinomial distribution with parameters $N_{k,\ell}$ and $(1/N_k)_{i\in[N_k]}$. Assume furthermore that the $\nu_{k,\ell}$, $k,\ell\in E$, are independent. We call this model the \emph{multi-type Wright--Fisher model}. In this case the joint distribution of the offspring sizes is given by
\begin{equation} \label{wfmjoint}
   \pr\bigg(
      \bigcap_{k,\ell\in E}\bigcap_{i=1}^{N_k}\{\nu_{k,\ell,i}=n_{k,\ell,i}\}
   \bigg)
   \ =\ \prod_{k,\ell\in E}
        \frac{N_{k,\ell}!}{\prod_{i=1}^{N_k}n_{k,\ell,i}!}
        \bigg(\frac{1}{N_k}\bigg)^{N_{k,\ell}}
\end{equation}
for all $n_{k,\ell,i}\in\nz_0$ with $\sum_{i=1}^{N_k}n_{k,\ell,i}=N_{k,\ell}$ for all $k,\ell\in E$. It is readily seen from (\ref{wfmjoint}) that (A), (A1) and (A2) hold.

Any multinomial random variable $X:=(X_s)_{s\in S}$ with at most countable set $S$ and parameters $n\in\nz_0$ and $(p_s)_{s\in S}$ has joint descending factorial moments $\me(\prod_{s\in S}(X_s)_{i_s})
=(n)_{\sum_{s\in S}i_s}\prod_sp_s^{i_s}$ for all $(i_s)_{s\in S}\in\nz_0^S$. Thus, the transition probability (\ref{rtrans2}) of the ancestral process takes the form
\begin{eqnarray}
   p_{\pi,\pi'}
   & = & \prod_{k\in E}\bigg(
            \frac{(N_k)_{j_k}}{(N_k)_{i_k}}
            \prod_{\ell\in E}
            \me\bigg(\prod_{s=1}^{j_k}(\nu_{k,\ell,s})_{i_{k,\ell,s}}\bigg)
         \bigg)\nonumber\\
   & = & \prod_{k\in E}\bigg(
         \frac{(N_k)_{j_k}}{(N_k)_{i_k}}
         \prod_{\ell\in E}
         \frac{(N_{k,\ell})_{i_{k,\ell}}}{N_k^{i_{k,\ell}}}\bigg)
   \ =:\ \Phi_j(T), \label{wfmtrans}
\end{eqnarray}
where $i_{k,\ell}:=\sum_{s=1}^{j_k}i_{k,\ell,s}$ for all $k,\ell\in E$. In particular, the diagonal entries are
\begin{equation} \label{wfmdiagonal}
   p_{\pi,\pi}\ =\ \prod_{k\in E}\frac{(N_{k,k})_{i_k}}{N_k^{i_k}},\qquad \pi\in{\cal P}_{n,E},
\end{equation}
since for $\pi=\pi'$, $i_{k,\ell,s}=\delta_{k,\ell}$ (Kronecker symbol) and, hence, $i_{k,\ell}=j_k\delta_{k,\ell}$, $k,\ell\in E$. For the single-type case $E=\{1\}$, (\ref{wfmtrans}) reduces to the well-known expression $p_{\pi,\pi'}=(N)_j/N^i$ for the classical (single-type) Wright--Fisher model with $N:=N_1$, $i:=i_1$ and $j:=j_1$. For the two-type case $E=\{1,2\}$ and sample size $n=2$, the state space ${\cal P}_{n,E}={\cal P}_{2,\{1,2\}}$ consists of the six states
\[
\begin{array}{lll}
   \pi_1\ :=\ \{(\{1,2\},1)\}, &
   \pi_2\ :=\ \{(\{1,2\},2)\}, &
   \pi_3\ :=\ \{(\{1\},1),(\{2\},1)\},\\
   \pi_4\ :=\ \{(\{1\},1),(\{2\},2)\}, &
   \pi_5\ :=\ \{(\{1\},2),(\{2\},1)\}, &
   \pi_6\ :=\ \{(\{1\},2),(\{2\},2)\}.
\end{array}
\]
In this case, by (\ref{wfmtrans}), the transition matrix $P=(p_{\pi_i,\pi_j})_{1\le i,j\le 6}$ of the ancestral process is given by
\[
P\ =\ \left(
   \begin{array}{cc|cccc}
   m_{11} & m_{12} & 0 & 0 & 0 & 0\\
   m_{21} & m_{22} & 0 & 0 & 0 & 0\\
   \hline
   \frac{(N_{1,1})_2}{N_1(N_1)_2} & \frac{(N_{2,1})_2}{(N_1)_2N_2} & \frac{(N_{1,1})_2}{N_1^2} & \frac{N_{1,1}N_{2,1}}{(N_1)_2} & \frac{N_{1,1}N_{2,1}}{(N_1)_2} & \frac{(N_2)_2(N_{2,1})_2}{(N_1)_2N_2^2}\\
   \frac{N_{1,1}N_{1,2}}{N_1^2N_2} & \frac{N_{2,1}N_{2,2}}{N_1N_2^2} & \frac{(N_1)_2N_{1,1}N_{1,2}}{N_1^3N_2} & \frac{N_{1,1}N_{2,2}}{N_1N_2} & \frac{N_{1,2}N_{2,1}}{N_1N_2} & \frac{(N_2)_2N_{2,1}N_{2,2}}{N_1N_2^3}\\
   \frac{N_{1,1}N_{1,2}}{N_1^2N_2} & \frac{N_{2,1}N_{2,2}}{N_1N_2^2} & \frac{(N_1)_2N_{1,1}N_{1,2}}{N_1^3N_2} & \frac{N_{1,2}N_{2,1}}{N_1N_2} & \frac{N_{1,1}N_{2,2}}{N_1N_2} & \frac{(N_2)_2N_{2,1}N_{2,2}}{N_1N_2^3}\\
   \frac{(N_{1,2})_2}{N_1(N_2)_2} & \frac{(N_{2,2})_2}{N_2(N_2)_2} & \frac{(N_1)_2(N_{1,2})_2}{N_1^2(N_2)_2} & \frac{N_{1,2}N_{2,2}}{(N_2)_2} & \frac{N_{1,2}N_{2,2}}{(N_2)_2} & \frac{(N_{2,2})_2}{N_2^2} \\
   \end{array}
\right),
\]
where $m_{k,\ell}:=N_{\ell,k}/N_k$, $k,\ell\in E$, denote the backward mutation probabilities.

The coalescence probabilities (\ref{coal1}) and (\ref{coal2}) reduce to
\begin{equation} \label{wfmcoal1}
   c_{k,\ell}(N_k,N_\ell)
   \ =\ \frac{N_k}{(N_\ell)_2}\me((\nu_{k,\ell,1})_2)
   \ =\ \frac{(N_{k,\ell})_2}{N_k(N_\ell)_2},\qquad k,\ell\in E,
\end{equation}
and, for $k,\ell_1,\ell_2\in E$ with $\ell_1\ne\ell_2$,
\begin{eqnarray}
   c_{k,\ell_1,\ell_2}(N_k,N_{\ell_1},N_{\ell_2})
   & = & \frac{N_k}{N_{\ell_1}N_{\ell_2}}
         \me(\nu_{k,\ell_1,1})\me(\nu_{k,\ell_2,1})\nonumber\\
   & = & \frac{N_k}{N_{\ell_1}N_{\ell_2}}
         \frac{N_{k,\ell_1}}{N_k}\frac{N_{k,\ell_2}}{N_k}
   \ = \ \frac{N_{k,\ell_1}N_{k,\ell_2}}{N_kN_{\ell_1}N_{\ell_2}}.
   \label{wfmcoal2}
\end{eqnarray}
Clearly, the consistency equation (\ref{consis}) holds by Proposition \ref{consisprop}. Alternatively, (\ref{consis}) is verified directly as follows. From (\ref{wfmtrans}) it is readily seen that
\[
\Phi_{j+e_k}(T(k,\ell))
\ =\ \Phi_j(T)\frac{N_k-j_k}{N_\ell-i_\ell}\frac{N_{k,\ell}-i_{k,\ell}}{N_k}
\ =\ \frac{\Phi_j(T)}{N_\ell-i_\ell}\bigg(1-\frac{j_k}{N_k}\bigg)(N_{k,\ell}-i_{k,\ell})
\]
and
\[
\Phi_j(T(k,\ell,s))
\ =\ \Phi_j(T)\frac{1}{N_\ell-i_\ell}\frac{N_{k,\ell}-i_{k,\ell}}{N_k}
\ =\ \frac{\Phi_j(T)}{N_\ell-i_\ell}\frac{1}{N_k}(N_{k,\ell}-i_{k,\ell}).
\]
In particular, the latter expression does not depend on $s$. Therefore,
\begin{eqnarray*}
   &   & \hspace{-15mm}\sum_{k\in E}\Phi_{j+e_k}(T(k,\ell))
         +\sum_{k\in E}\sum_{s=1}^{j_k} \Phi_j(T(k,\ell,s))\\
   & = & \frac{\Phi_j(T)}{N_\ell-i_\ell}\sum_{k\in E}
         \bigg(1-\frac{j_k}{N_k}\bigg)(N_{k,\ell}-i_{k,\ell})
         + \frac{\Phi_j(T)}{N_\ell-i_\ell}\sum_{k\in E} \frac{j_k}{N_k}(N_{k,\ell}-i_{k,\ell})\\
   & = & \frac{\Phi_j(T)}{N_\ell-i_\ell}\sum_{k\in E}(N_{k,\ell}-i_{k,\ell})
   \ = \ \frac{\Phi_j(T)}{N_\ell-i_\ell}(N_\ell-i_\ell)
   \ = \ \Phi_j(T),
\end{eqnarray*}
which is (\ref{consis}).

\subsubsection*{A weak mutation example} \label{weak}
In the following all convergence statements and all asymptotic relations are meant as $N_{\rm min}:=\min_{k\in E}N_k\to\infty$. Assume that $N_{k,k}/N_k\to 1$ as $N_{\rm min}\to\infty$ for all $k\in E$ or, equivalently, that $N_{k,\ell}/N_\ell\to 0$ as $N_{\rm min}\to\infty$ for all $k,\ell\in E$ with $k\ne\ell$. Note that this corresponds to a weak mutation regime. From $\me(\nu_{k,k,1})=N_{k,k}/N_k\to 1$ and $\me(\nu_{k,k,1}\nu_{k,k,2})=(N_{k,k})_2/N_k^2\to 1$ for all $k\in E$ we conclude (Lemma \ref{identitylemma}) that $P_N\to I$.
Let us also assume that the subpopulation sizes $N_k$ satisfy the calibration assumption
\begin{equation} \label{calibration}
   \frac{N_{\rm min}}{N_k}\ \to\ a_k
\end{equation}
as $N_{\rm min}\to\infty$ for some given real constants $a_k\ge 0$. If $a_k>0$ for all $k\in E$ then (\ref{calibration}) means that all subpopulation sizes are asymptotically of the same order. From (\ref{wfmcoal1}) and (\ref{wfmcoal2}) it follows for all $k,\ell\in E$ that $N_kc_{k,\ell}(N_k,N_\ell)\to\delta_{k,\ell}$ as $N_{\rm min}\to\infty$
and for all $k,\ell_1,\ell_2\in E$ with $\ell_1\ne\ell_2$ that $N_kc_{k,\ell_1,\ell_2}(N_k,N_{\ell_1},N_{\ell_2})\to 0$ as $N_{\rm min}\to\infty$. Recall that $N:=(N_k)_{k\in E}$. The natural choice for the time scaling $c_N$ is
\begin{equation} \label{wfmscaling}
   c_N\ :=\ \frac{1}{N_{\rm min}}\ =\ \frac{1}{\min_{k\in E}N_k}.
\end{equation}
Comparing the coalescence probability (\ref{wfmcoal1}) with $c_N$ we have
\[
\frac{c_{k,\ell}(N_k,N_\ell)}{c_N}\ =\ N_{\min}c_{k,\ell}(N_k,N_\ell)
\ \sim\ a_kN_k c_{k,\ell}(N_k,N_\ell)\ \to\ a_k\delta_{k,\ell},
\]
which shows that, in the limit, binary mergers of two $\ell$-individuals to a $k$-individual are only possible for $k=\ell$ (at the rate $a_k$). Similarly, for $k,\ell_1,\ell_2\in E$ with $\ell_1\ne\ell_2$,
\[
\frac{c_{k,\ell_1,\ell_2}(N_k,N_{\ell_1},N_{\ell_2})}{c_N}
\ =\ a_kN_kc_{k,\ell_1,\ell_2}(N_k,N_{\ell_1},N_{\ell_2})
\ \to\ 0,
\]
showing that binary mergers of individuals of different types are impossible in the limit. Having the monotonicity (Corollary \ref{monotonecorollary})) in mind, it follows that the assumptions of Theorem \ref{main2} hold with all the limits $\phi_j(T)$ equal to zero except for $\phi_{e_k}(T)=a_k$ in case of a binary merging of two $k$-individuals to a single $k$-individual. Thus, Theorem \ref{main2} is applicable. The limiting Markov chain $\Pi^{(n)}$ in Theorem \ref{main2} is a multi-type Kingman $n$-coalescent in the sense that binary mergers of two ancestral lineages of the same type $k$ to a single $k$-individual occur at the rate $a_k$. In particular, during any binary merging event a change of type is impossible. The measures $Q_j$ in Proposition \ref{int1prop} are all equal to the zero measure except for $Q_{e_k}$ being a Dirac measure on $\Delta_{e_k}$ assigning mass $a_k$ to the single point $0\in\Delta_{e_k}$, $k\in E$. The measure $\Xi$ in Theorem \ref{main3} is the zero measure on $(\Delta\setminus\{0\})\times E^\nz$. If $a_k=0$, then there is no activity in subpopulation $k$, which could be interpreted as a sort of seed bank. For general information on dormancy and seed bank models we refer the reader to the survey article \cite{BlathKurtSlowikWilkeBerenguer2024} and the references therein. Similar multi-type Kingman ($n$-)coalescents, even with additional mutation forces, occur for example in \cite[p.~4219]{JohnstonKyprianouRogers2023} and \cite[p.~107]{Moehle2024}.

\subsubsection*{A strong mutation example} \label{strong}
Let $M\in\nz$ and assume that $N_{k,\ell}=M$ for all $k,\ell\in E$, corresponding to strong mutation. Then, $N_k=dM$, $k\in E$ and, hence, $N_{\rm min}:=\min_{k\in E}N_k=dM$. Both coalescence probabilities, (\ref{wfmcoal1}) and (\ref{wfmcoal2}), are asymptotically equal to $1/(dM^3)=1/(d^2N_{\rm min})$. Therefore, a reasonable time-scaling is $c_N:=1/N_{\rm min}$. From (\ref{wfmtrans}) it follows for all $\pi,\pi'\in{\cal P}_{n,E}$ with $\pi\subseteq\pi'$ that
\begin{eqnarray*}
   p_{\pi,\pi'}
   & = & \frac{\prod_{k\in E}(N_{\rm min})_{j_k}}{\prod_{\ell\in E}(N_{\rm min})_{i_\ell}}
\frac{\prod_{k,\ell\in E}(N_{\rm min}/d)_{i_{k,\ell}}}{\prod_{\ell\in E}N_{\rm min}^{i_\ell}}\\
   & = & \frac{N_{\rm  min}^{|\pi'|}\Big(1-c_N\sum_{k\in E}\binom{j_k}{2}+o(c_N)\Big)}{N^{|\pi|}\Big(1-c_N\sum_{\ell\in E}\binom{i_\ell}{2}+o(c_N)\Big)}
         \frac{(\frac{N_{\rm min}}{d})^{|\pi|}\Big(1-\frac{d}{N}\sum_{k,\ell\in E}\binom{i_{k,\ell}}{2}+o(c_N)\Big)}{N_{\rm min}^{|\pi|}}\\
   & = & \frac{N_{\rm min}^{|\pi'|-|\pi|}}{d^{|\pi|}}
         \big(1+c_N \kappa(\pi,\pi')+o(c_N)\big)
\end{eqnarray*}
with $\kappa(\pi,\pi')\ :=\ \sum_{\ell\in E}\binom{i_\ell}{2}-\sum_{k\in E}\binom{j_k}{2}-d\sum_{k,\ell\in E}\binom{i_{k,\ell}}{2}$. Thus, the transition matrix $P_N$ of the ancestral process satisfies $P_N=A+c_NB+o(c_N)$ as $N_{\rm min}\to\infty$, where the matrices $A:=(a_{\pi,\pi'})_{\pi,\pi'\in{\cal P}_{n,E}}$ and $B:=(b_{\pi,\pi'})_{\pi,\pi'\in{\cal P}_{n,E}}$ have entries
\[
a_{\pi,\pi'}\ :=\ \left\{
\begin{array}{cl}
   d^{-|\pi|} & \mbox{if $r_n(\pi)=r_n(\pi')$},\\
   0 & \mbox{otherwise,}
\end{array}
\right.
\]
and
\[
b_{\pi,\pi'}\ :=\ \left\{
\begin{array}{cl}
   d^{-|\pi|}\kappa(\pi,\pi') & \mbox{if $r_n(\pi)=r_n(\pi')$,}\\
   d^{-|\pi|} & \mbox{if $r_n(\pi)\subset r_n(\pi')$ and $|\pi'|=|\pi|-1$},\\
   0 & \mbox{otherwise.}
\end{array}
\right.
\]
The matrix $A$ is stochastic and a projection ($A^2=A$). By Theorem \ref{main1}, the ancestral process $({\cal A}_r^{(n,N)})_{r\in\nz_0}$ converges in $D_{{\cal P}_{n,E}}(\nz_0)$ as $N_{\rm min}\to\infty$ to a limiting discrete-time Markov chain $(\Pi_r^{(n)})_{r\in\nz_0}$ with transition matrix $A$. No merging events occur in this limit. In each time step only the types of the blocks may change. More precisely, in each time step the resulting type of every block is chosen at random and independently of everything else. In this case the support of the limiting function $\phi_j$ in Theorem \ref{main1} is concentrated on tensors $T:=(t_{k,\ell})_{k,\ell\in E}$, where the entries $t_{k,\ell}:=(i_{k,\ell,s})_{s\in[j_k]}$ satisfy the constrains
\begin{equation} \label{tensorconstrains}
\begin{minipage}[c]{10cm}
\begin{itemize}
   \item $i_{k,\ell,s}\in\{0,1\}$ for all $k,\ell\in E$ and $s\in [j_k]$,
   \item $\sum_{k\in E}\sum_{s=1}^{j_k}i_{k,\ell,s}=j_\ell$ for all $\ell\in E$ and
   \item $\sum_{\ell\in E}\sum_{s=1}^{j_\ell}i_{k,\ell,s}=j_k$ for all $k\in E$.
\end{itemize}
\end{minipage}
\end{equation}
In particular, the support of $\phi_j$ is not concentrated on diagonal tensors. For $E=\{1,2,3\}$, an example of a transition step with $j=(j_1,j_2,j_3)=(4,5,7)$ and corresponding tensor
\begin{equation} \label{tensorexample}
T=\left(
\begin{array}{ccc}
   (1,0,0,0) & (0,1,1,0) & (0,0,0,1)\\
   (1,0,0,0,0) & (0,0,0,0,0) & (0,1,1,1,1)\\
   (1,1,0,0,0,0,0) & (0,0,1,1,1,0,0) & (0,0,0,0,0,1,1)
\end{array}
\right)
\end{equation}
is illustrated in Figure \ref{figure2}.
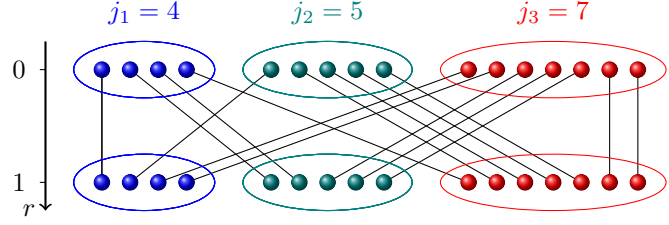
\begin{figure}[htpb]
   \caption[]{A transition step of the limiting process $\Pi$ for $E=\{1,2,3\}$ and $j=(4,5,7)$ corresponding to the tensor $T$ in (\ref{tensorexample})}
   \centering
   \definecolor{myblue}{HTML}{92dcec}
\begin{tikzpicture}[scale=0.75]
   \def \N {16} 
   \def \r {1}
   \def \radius {0.15} 
   
   \draw[->,thick] (-0.5,2*\r+0.5) -- (-0.5,0-0.5) node[below,left] {$r$};
   
   \foreach \y in {0,...,\r}
      \draw (-0.6,2*\r-2*\y) -- (-0.5,2*\r-2*\y) node[left=4pt] {\y};
   
   \draw[semithick] (0.5,0) -- (0.5,2);
   \draw[thin] (1,0) -- (3.5,2);
   \draw[thin] (1.5,0) -- (7,2);
   \draw[thin] (2,0) -- (7.5,2);
   \draw[thin] (3.5,0) -- (1,2);
   \draw[thin] (4,0) -- (1.5,2);
   \draw[thin] (4.5,0) -- (8,2);
   \draw[thin] (5,0) -- (8.5,2);
   \draw[thin] (5.5,0) -- (9,2);
   \draw[thin] (7,0) -- (2,2);
   \draw[thin] (7.5,0) -- (4,2);
   \draw[thin] (8,0) -- (4.5,2);
   \draw[thin] (8.5,0) -- (5,2);
   \draw[thin] (9,0) -- (5.5,2);
   \draw[thin] (9.5,0) -- (9.5,2);
   \draw[thin] (10,0) -- (10,2);
   
   \draw[color=blue] (1.25,2*\r+1) node {\centering $j_1=4$};
   \foreach \n in {0,...,\r} \foreach \i in {1,...,4}
      {
      \draw[color=blue] (1.25,2*\n) ellipse (1.25cm and 5mm);
      \shade[ball color=blue] (\i/2,2*\n) circle (\radius);
      }
   \draw[color=teal] (4.5,2*\r+1) node {\centering $j_2=5$};
   \foreach \n in {0,...,\r} \foreach \i in {1,...,5}
      {
      \draw[color=teal] (4.5,2*\n) ellipse (1.5cm and 5mm);
      \shade[ball color=teal] (3+\i/2,2*\n) circle (\radius);
      }
   \draw[color=red] (8.5,2*\r+1) node {\centering $j_3=7$};
   \foreach \n in {0,...,\r} 
      {
      \draw[color=red] (8.5,2*\n) ellipse(2cm and 5mm);
      
      \foreach \i in {1,...,7} \shade[ball color=red] (6.5+\i/2,2*\n) circle (\radius);
      }
\end{tikzpicture}
   \label{figure2}
\end{figure}

Having the expansion $P_N=A+c_NB+o(c_N)$ in mind, the ancestral structure can be analysed in more detail as follows. By \cite[Theorem 1]{Moehle1998}, the finite-dimensional distributions of the time-scaled ancestral process $({\cal A}_{\lfloor t/c_N\rfloor}^{(n,N)})_{t\ge 0}$ converge as $N_{\rm min}\to\infty$ to those of a continuous-time limiting Markov process $(\widetilde{\Pi}_t^{(n)})_{t\ge 0}$  with state space ${\cal P}_{n,E}$
and transition matrix $Ae^{tG}$, $t>0$, with infinitesimal generator $G:=ABA$. For example, if $n=2$ and $d=2$, then $N_{\rm min}=2M$,
\[
P_N\ =\ \left(
   \begin{array}{cc|cccc}
      \frac{1}{2} & \frac{1}{2} & 0 & 0 & 0 & 0\\
      \frac{1}{2} & \frac{1}{2} & 0 & 0 & 0 & 0\\
      \hline
      \frac{M-1}{4M(2M-1)} & \frac{M-1}{4M(2M-1)} & \frac{M-1}{4M} & \frac{M}{2(2M-1)} & \frac{M}{2(2M-1)} & \frac{M-1}{4M}\\
      \frac{1}{8M} & \frac{1}{8M} & \frac{2M-1}{8M} & \frac{1}{4} & \frac{1}{4}& \frac{2M-1}{8M}\\
      \frac{1}{8M} & \frac{1}{8M} & \frac{2M-1}{8M} & \frac{1}{4} & \frac{1}{4} & \frac{2M-1}{8M}\\
      \frac{M-1}{4M(2M-1)} & \frac{M-1}{4M(2M-1)} & \frac{M-1}{4M} & \frac{M}{2(2M-1)} & \frac{M}{2(2M-1)} & \frac{M-1}{4M}\\
   \end{array}
\right).
\]
and $P_N=A+c_NB+o(c_N)$ as $N_{\rm min}\to\infty$, where
\[
A\ :=\ \left(
   \begin{array}{cc|cccc}
      \frac{1}{2} & \frac{1}{2} & 0 & 0 & 0 & 0\\
      \frac{1}{2} & \frac{1}{2} & 0 & 0 & 0 & 0\\
      \hline
      0 & 0 & \frac{1}{4} & \frac{1}{4} & \frac{1}{4} & \frac{1}{4}\\
      0 & 0 & \frac{1}{4} & \frac{1}{4} & \frac{1}{4} & \frac{1}{4}\\
      0 & 0 & \frac{1}{4} & \frac{1}{4} & \frac{1}{4} & \frac{1}{4}\\
      0 & 0 & \frac{1}{4} & \frac{1}{4} & \frac{1}{4} & \frac{1}{4}\\
   \end{array}
\right)
\quad\mbox{and}\quad
B\ :=\
\left(
   \begin{array}{cc|cccc}
      0 & 0 & 0 & 0 & 0 & 0\\
      0 & 0 & 0 & 0 & 0 & 0\\
      \hline
      \frac{1}{4} & \frac{1}{4} & -\frac{1}{2} & \frac{1}{4} & \frac{1}{4} & -\frac{1}{2}\\
      \frac{1}{4} & \frac{1}{4} & -\frac{1}{4} & 0 & 0 & -\frac{1}{4}\\
      \frac{1}{4} & \frac{1}{4} & -\frac{1}{4} & 0 & 0 & -\frac{1}{4}\\
      \frac{1}{4} & \frac{1}{4} & -\frac{1}{2} & \frac{1}{4} & \frac{1}{4} & -\frac{1}{2}\\
   \end{array}
\right).
\]
We refer the reader to H\"ossjer \cite[Theorem 4.1]{Hoessjer2011} for a similar convergence result for the process counting the total number of ancestral lineages. Related convergence statements and proof methods based on Lemma 1 of \cite{Moehle1998} can be found in Kaj et al. \cite{KajKroneLascoux2001}, Nordborg and Krone \cite{NordborgKrone2002} and Sagitov and Jagers \cite{SagitovJagers2005}. We leave similar convergence results for the associated multi-type block counting process for the interested reader.

One may modify the multi-type Wright--Fisher model by assuming in addition that the $N_{k,\ell}$ are random. For example, one may assume that, for each $\ell\in E$, the vector $(N_{k,\ell})_{k\in E}$ has a multinomial distribution with parameters $N_\ell$ and $(u_{k,\ell})_{k\in E}$ for some given parameters $u_{k,\ell}$ and that these vectors are independent over $\ell\in E$. We leave the technical analysis of this modified Wright--Fisher model for the interested reader.

\subsubsection{A multi-type mutation model without merging events} \label{mutationmodel}
Suppose that in each time step a given number $N_{k,\ell}\in\nz_0$ of individuals of type $k\in E$ mutate to type $\ell\in E$, where the numbers $N_{k,\ell}$, $k,\ell\in E$, satisfy the constrains
\begin{equation} \label{mutation_constrains}
   \sum_{\ell\in E}N_{k,\ell}\ =\ N_k
   \ =\ \sum_{\ell\in E}N_{\ell,k}\ \in\ \nz,\qquad k\in E.
\end{equation}
This mutation model can be formulated in terms of a model with offspring sizes $\nu_{k,\ell,i}$ as described in Section \ref{intro} using the interpretation that `mutation to type $\ell$' is equivalent to `having one offspring of type $\ell$'. Since, with this interpretation, each individual has exactly one offspring, the family sizes $\nu_{k,\ell,i}$ satisfy
\[
\sum_{\ell\in E}\nu_{k,\ell,i}\ =\ 1,\qquad k\in E,i\in[N_k].
\]
In particular, each $\nu_{k,\ell,i}$ takes values in $\{0,1\}$. In contrast to the Wright--Fisher model, for arbitrary but fixed $k\in E$ and $i\in[N_k]$, the offspring sizes $\nu_{k,\ell,i}$, $\ell\in E$, are not independent. The joint distribution of the offspring sizes is obtained as follows. For $k\in E$ define $\nu_k:=(\nu_{k,\ell,i})_{i\in[N_k],\ell\in E}$. It is readily seen that the $\nu_k$, $k\in E$, are independent, where each $\nu_k$ is uniformly distributed on the set $\Omega_k$ consisting of all $(n_{k,\ell,i})_{\ell\in E,i\in[N_k]}$ satisfying
\begin{itemize}
   \item $n_{k,\ell,i}\in\{0,1\}$ for all $\ell\in E$ and $i\in[N_k]$,
   \item $\sum_{i=1}^{N_k}n_{k,\ell,i}=N_{k,\ell}$ for all $\ell\in E$, and
   \item $\sum_{\ell\in E}n_{k,\ell,i}=1$ for all $i\in[N_k]$.
\end{itemize}
Note that $|\Omega_k|=N_k!/\prod_{\ell\in E} N_{k,\ell}!$. Thus, the joint offspring distribution is given by
\begin{equation} \label{mutationjoint}
   \pr\bigg(\bigcap_{k,\ell\in E}\bigcap_{i=1}^{N_k}
   \{\nu_{k,\ell,i}=n_{k,\ell,i}\}\bigg)
   \ =\ \prod_{k\in E}\frac{1}{|\Omega_k|}
   \ =\ \prod_{k\in E}\frac{\prod_{\ell\in E}N_{k,\ell}!}{N_k!}
\end{equation}
for all $n_{k,\ell,i}\in\{0,1\}$, $k,\ell\in E$, $i\in[N_k]$, satisfying $\sum_{i=1}^{N_k}n_{k,\ell,i}=N_{k,\ell}$ for all $k,\ell\in E$ and $\sum_{\ell\in E}n_{k,\ell,i}=1$ for all $k\in E$ and $i\in [N_k]$. In particular, Assumptions (A), (A1) and (A2) hold.

Looking backwards in time no merging events are possible. Thus, a transition of the discrete ancestral process from $\pi\in{\cal P}_{n,E}$ to $\pi'\in{\cal P}_{n,E}$ is only possible if $\pi'$ has the same blocks as $\pi$ (types of the blocks disregarded here). In this case, by (\ref{rtrans2}),
\begin{eqnarray}
   p_{\pi,\pi'}
   & = & \prod_{k\in E}\bigg(\frac{(N_k)_{j_k}}{(N_k)_{i_k}}
         \me\bigg(
            \prod_{\ell\in E}\prod_{s=1}^{j_k}(\nu_{k,\ell,s})_{i_{k,\ell,s}}
         \bigg)\bigg)\nonumber\\
   & = & \prod_{k\in E}\bigg(\frac{(N_k)_{j_k}}{(N_k)_{i_k}}
         \pr\bigg(
            \bigcap_{\ell\in E}
            \bigcap_{\substack{s=1\\ i_{k,\ell,s}=1}}^{j_k}\{\nu_{k,\ell,s}=1\}
         \bigg)\bigg)
   \ = \ \prod_{k\in E}\bigg(\frac{(N_k)_{j_k}}{(N_k)_{i_k}}
         \frac{\frac{(N_k-j_k)!}{\prod_{\ell\in E}(N_{k,\ell}-i_{k,\ell})!}}
         {|\Omega_k|}\bigg)\nonumber\\
   & = & \prod_{k\in E}
         \frac{\prod_{\ell\in E}(N_{k,\ell})_{i_{k,\ell}}}{(N_k)_{i_k}}
   \ = \ \prod_{\ell\in E}
         \frac{\prod_{k\in E}(N_{k,\ell})_{i_{k,\ell}}}{(N_\ell)_{i_\ell}},
         \label{mutation_trans}
\end{eqnarray}
where, for $k,\ell\in E$, $i_{k,\ell}:=\sum_{s=1}^{j_k}i_{k,\ell,s}$ denotes the number of blocks being a $k$-block of $\pi'$ and an $\ell$-block of $\pi$. Note that $i_\ell:=\sum_{k\in E}i_{k,\ell}$ is the number of $\ell$-blocks of $\pi$, $\ell\in E$, and that $j_k:=\sum_{\ell\in E}i_{k,\ell}$ is the number of $k$-blocks of $\pi'$. In particular, the transition matrix of the ancestral process has diagonal entries
\begin{equation}
   p_{\pi,\pi}\ =\ \sum_{k\in E}\frac{(N_{k,k})_{i_k}}{(N_k)_{i_k}},\qquad \pi\in{\cal P}_{n,E}.
\end{equation}
This model is a reformulation of the migration step considered around Figure 2 on p.~105 of \cite{Moehle2024} in terms of a multi-type Cannings model with offspring distribution (\ref{mutationjoint}). A migration from colony $k$ to colony $\ell$ in the model in \cite{Moehle2024} is interpreted as a mutation from type $k$ to type $\ell$. Note that (\ref{mutation_trans}) is in agreement with \cite[Eq.~(12)]{Moehle2024}. Figure 2 of \cite{Moehle2024} shows a graphical representation of one generation step of this model for $E=\{1,2,3\}$, subpopulation sizes $N_1:=4$, $N_2:=5$ and $N_3:=7$, and mutation numbers $N_{1,1}:=1$, $N_{1,2}:=2$, $N_{1,3}:=1$, $N_{2,1}:=1$, $N_{2,2}:=0$, $N_{2,3}:=4$, $N_{3,1}:=2$, $N_{3,2}:=3$ and $N_{3,3}:=2$.

Clearly, the transition probabilities (\ref{mutation_trans}) satisfy $\sum_{\pi'\in{\cal P}_{n,E}}p_{\pi,\pi'}=1$ for all $\pi\in{\cal P}_{n,E}$, which can be also formally verified as follows. Fix $\pi\in{\cal P}_{n,E}$. For $\ell\in E$ let $i_\ell$ denote the number of $\ell$-blocks of $\pi$. For any given matrix $A:=(i_{k,\ell})_{k,\ell\in E}$ with nonnegative integer entries $i_{k,\ell}\in\nz_0$ satisfying $\sum_{k\in E}i_{k,\ell}=i_\ell$ for all $\ell\in E$ let us denote with ${\cal P}(A)$ the subset of all $\pi'\in{\cal P}_{n,E}$ having the same (un-labeled) blocks as $\pi$ and with the property that the number of blocks being a $k$-block of $\pi'$ and an $\ell$-block of $\pi$ is equal to $i_{k,\ell}$. Note that ${\cal P}(A)$ depends on $n,E$ and $\pi$, but this dependence is suppressed in our notation, since it is not important in the following. By (\ref{mutation_trans}),
\[
\sum_{\pi'\in{\cal P}_{n,E}}p_{\pi,\pi'}
\ =\ \sum_A\sum_{\pi'\in {\cal P}(A)} p_{\pi,\pi'}
\ =\ \sum_A |{\cal P}(A)|\prod_{\ell\in E}
     \frac{\prod_{k\in E}(N_{k,\ell})_{i_{k,\ell}}}{(N_\ell)_{i_\ell}},
\]
where the sum $\sum_A$ extents over all matrices $A=(i_{k,\ell})_{k,\ell\in E}$ as described above. Since $|{\cal P}(A)|=\prod_{\ell\in E}\frac{i_\ell!}{\prod_{k\in E}i_{k,\ell}!}$, it follows that
\[
\sum_{\pi'\in{\cal P}_{n,E}}p_{\pi,\pi'}
\ = \ \sum_A \prod_{\ell\in E}
      \frac{\prod_{k\in E}\binom{N_{k,\ell}}{i_{k,\ell}}}{\binom{N_\ell}{i_\ell}}
\ =\ \prod_{\ell\in E} \bigg(
     \sum_{\substack{(i_{k,\ell})_{k\in E}\in\nz_0^E\\ \sum_{k\in E}i_{k,\ell}=i_\ell}}
       \frac{\prod_{k\in E}\binom{N_{k,\ell}}{i_{k,\ell}}}{\binom{N_\ell}{i_\ell}}
    \bigg)
   \ =\ 1,
\]
since the sum inside the last brackets is equal to $1$ (total mass of a multi-hypergeometric distribution).
%
%
%
%

For every $j=(j_k)_{k\in E}\in\nz_0^E$, the support of the function $\Phi_j$ is concentrated on tensors $T$ satisfying the constrains (\ref{tensorconstrains}). Due to these constrains, the general consistency equations in Proposition \ref{consisprop} simplify to
\[
\Phi_j(T)\ =\ \Phi_{j+e_\ell}(T(\ell,\ell))+\sum_{s=1}^{j_\ell}\Phi_j(T(\ell,\ell,s)),
\qquad\ell\in E.
\]
These consistency equations are, except that we are in the discrete model before taking any limit, of the reduced form (\ref{consisdiag}), although the support of the functions $\Phi_j$ is not concentrated on diagonal tensors.

If, for all $k,\ell\in E$, $N_{k,\ell}/N_\ell\to\rho_{k,\ell}$ as $N_{\rm min}\to\infty$ for some constants $\rho_{k,\ell}$, then Theorem \ref{main1} is applicable, since the transition probability in (\ref{mutation_trans}) converges to $\prod_{k,\ell\in E}\rho_{k,\ell}=:a_{\pi,\pi'}$ as $N_{\rm min}\to\infty$. The limiting process $(\Pi_r^{(n)})_{r\in\nz_0}$ in Theorem \ref{main1} has transition probabilities $a_{\pi,\pi'}$ whenever $\pi$ and $\pi'$ have the same blocks (types of the blocks disregarded here).

Since all coalescence probabilities (\ref{coal1}) and (\ref{coal2}) are equal to zero in this model, Theorem \ref{main2} is not applicable with the standard choice (\ref{scaling}) for the time-scaling $c_N$.

\subsection{Proofs} \label{proofs}
We start this section with the proof of the formula (\ref{rtrans2}) for the transition probability $p_{\pi,\pi'}$.
\begin{proof}[Proof of Eq.~(\ref{rtrans2})]
   Let ${\cal G}$ denote the $\sigma$-algebra generated by all the offspring sizes $\nu_{k,\ell,i}$, $k,\ell\in E$, $i\in[N_k]$. Since, independently for each type $\ell\in E$, individuals of type $\ell$ in the child generation are randomly assigned to parental offspring lineages of type $\ell$, it follows conditional on the offspring sizes that
   \[
   \pr({\cal A}_r=\pi'\,|\,{\cal A}_{r-1}=\pi,{\cal G})
   \ =\ \sum_n\prod_{\ell\in E}
   \frac{\prod_{k\in E}\prod_{s=1}^{j_k} (\nu_{k,\ell,n(k,s)})_{i_{k,\ell,s}}}
   {(N_\ell)_{i_\ell}},
   \]
   where the sum $\sum_n$ extents over all $n:=(n(k,s))_{k\in E,s\in[j_k]}$ with $n(k,s)\in[N_k]$ for all $k\in E$ and $s\in[j_k]$ and $n(k,s)\ne n(k,r)$ for all $k\in E$ and $s,r\in[j_k]$ with $s\ne r$. Taking expectation yields
   \[
   p_{\pi,\pi'}\ =\ \frac{1}{\prod_{\ell\in E}(N_\ell)_{i_\ell}}
   \sum_n\me\bigg(\prod_{k,\ell\in E}
   \prod_{s=1}^{j_k} (\nu_{k,\ell,n(k,s)})_{i_{k,\ell,s}}
   \bigg).
   \]
   Exploiting the joint exchangeability (A) it follows that
   \[
   p_{\pi,\pi'}\ =\ \frac{1}{\prod_{\ell\in E}(N_\ell)_{i_\ell}}
   \sum_n\me\bigg(\prod_{k,\ell\in E}
   \prod_{s=1}^{j_k} (\nu_{k,\ell,s})_{i_{k,\ell,s}}
   \bigg).
   \]
   The last expectation does not depend on $n$. Thus, (\ref{rtrans2}) follows from $\sum_n 1=\prod_{k\in E}(N_k)_{j_k}$.
\end{proof}
Next we provide a short and elegant proof of the crucial consistency relation (\ref{consis}).
\begin{proof}[Proof of Proposition \ref{consisprop}]
   Define the random variable $X:=\prod_{k,\ell\in E}\prod_{s\in[j_k]}(\nu_{k,\ell,s})_{i_{k,\ell,s}}$ and the two constants $A:=\prod_{k\in E}(N_k)_{j_k}$ and $B:=\prod_{\ell\in E}(N_\ell)_{i_\ell}$. Now fix $\ell\in E$ with $i_\ell<N_\ell$. For all $k\in E$ it follows from $N_{k,\ell}=\sum_{s=1}^{N_k}\nu_{k,\ell,s}$ and with the notation
   $i_{k,\ell}:=\sum_{s=1}^{j_k}i_{k,\ell,s}$ that
   \begin{eqnarray*}
      \me\big(X(N_{k,\ell}-i_{k,\ell})\big)
      & = & \me\bigg(X\sum_{s=j_k+1}^{N_k}\nu_{k,\ell,s}\bigg)
            + \sum_{s=1}^{j_k}\me\big(X(\nu_{k,\ell,s}-i_{k,\ell,s})\big)\\
      & = & (N_k-j_k)\me(X\nu_{k,\ell,j_k+1}) +   \sum_{s=1}^{j_k}\me\big(X(\nu_{k,\ell,s}-i_{k,\ell,s})\big)
   \end{eqnarray*}
   by the joint exchangeability (A). Summation over all $k\in E$ yields
   \[
   (N_\ell-i_\ell)\me(X)\ =\
   \sum_{k\in E}(N_k-j_k)\me(X\nu_{k,\ell,j_k+1})
   + \sum_{k\in E}\sum_{s=1}^{j_k}\me\big(X(\nu_{k,\ell,s}-i_{k,\ell,s})\big).
   \]
   Multiplication of both sides of this equation with $A/((N_\ell-i_\ell)B)$ $(>0)$ shows that
   \[
   \frac{A}{B}\me(X)\ =\
   \sum_{k\in E}\frac{(N_k-j_k)A}{(N_\ell-i_\ell)B}\me(X\nu_{k,\ell,j_k+1})
   + \sum_{k\in E}\sum_{s=1}^{j_k}
   \frac{A}{(N_\ell-i_\ell)B}\me\big(X(\nu_{k,\ell,s}-i_{k,\ell,s})\big),
   \]
   which is the desired Eq.~(\ref{consis}), since $\frac{A}{B}\me(X)=\Phi_j(T)$ by the definition (\ref{Phi}) of $\Phi_j$, and, similarly, $\frac{(N_k-j_k)A}{(N_\ell-i_\ell)B}\me(X\nu_{k,\ell,j_k+1})=\Phi_{j+e_k}\big(T(k,\ell)\big)$ and $\frac{A}{(N_\ell-i_\ell)B}\me\big(X(\nu_{k,\ell,s}-i_{k,\ell,s})\big)
   =\Phi_j\big(T(k,\ell,s)\big)$.
\end{proof}
\begin{proof}[Proof of Corollary \ref{monotonecorollary}]
   For $j=j'$, (\ref{monotone}) follows from the consistency equation (\ref{consis}). The result thus follows by induction on $\sum_{k\in E}(j'_k-j_k)$ and exploiting in each induction step the consistency equation (\ref{consis}) again.
\end{proof}
\begin{proof}[Proof of Corollary \ref{naturalcoupling}]
   The proof exploits ideas going back at least to Burke and Rosenblatt \cite{BurkeRosenblatt1958}. Since $\varrho_{n,m}=\varrho_{m+1,m}\circ\cdots\circ\varrho_{n,n-1}$, we can and do assume without loss of generality that $m=n-1$. Define $f:=\varrho_{n,m}$. Fix $\pi\in{\cal P}_{n,E}$ and $\tau'\in{\cal P}_{m,E}$. In the following an expression for
   \begin{equation} \label{localsum}
      \pr(f\circ{\cal A}_r^{(n)}=\tau'\,|\,{\cal A}_{r-1}^{(n)}=\pi)\ =\ \sum_{\pi'\in f^{-1}(\tau')}p_{\pi,\pi'}
   \end{equation}
   is derived, which in particular shows that (\ref{localsum}) depends on $\pi$ only via $\tau:=f(\pi)$.

   If $\tau\not\subseteq\tau'$, then $\pi\not\subseteq\pi'$, i.e., $p_{\pi,\pi'}=0$ for all $\pi'\in f^{-1}(\tau')$ and (\ref{localsum}) is equal to zero. Assume now that $\tau\subseteq\tau'$.

   Clearly, the partition $\tau'$ is of the form $\tau'=\{(B_1,k_1),\ldots,(B_a,k_a)\}$ with $a\in[m]$ blocks having types $k_1,\ldots,k_a\in E$ respectively. For $k\in E$ let $j_k:=|\{\alpha\in[a]:k_\alpha=k\}|$ denote the number of $k$-blocks of $\tau'$.

   Since $\tau\subseteq\tau'$, the partition $\tau$ has blocks $B_{\alpha\beta}$, $\alpha\in[a]$, $\beta\in[b_\alpha]$, satisfying $B_\alpha=\bigcup_{\beta=1}^{b_\alpha} B_{\alpha\beta}$ for all $\alpha\in[a]$. Moreover, each block $B_{\alpha\beta}$ of $\tau$ is labeled with some type $\ell_{\alpha\beta}\in E$. For $\ell\in E$ let $i_\ell:=|\{(\alpha,\beta):\ell_{\alpha\beta}=\ell\}|$ denote the number of $\ell$-blocks of $\tau$. Furthermore, for $k,\ell\in E$ and $s\in[j_k]$, let $i_{k,\ell,s}$ denote the groups size of $\ell$-blocks of $\tau$ merging to the $s$-th $k$-block of $\tau'$ and define the tensor $T:=(t_{k,\ell})_{k,\ell\in E}$ via $t_{k,\ell}:=(i_{k,\ell,s})_{s\in[j_k]}$ for all $k,\ell\in E$.

   For $k\in E$ define $\pi_0'(k):=\{(B_1,k_1),\ldots,(B_a,k_a),(\{n\},k)\}\in{\cal P}_{n,E}$ and for $\alpha\in[a]$ define $\pi_\alpha':=\{(B_1,k_1),\ldots,(B_\alpha\cup\{n\},k_\alpha),\ldots,(B_a,k_a)\}\in{\cal P}_{n,E}$.

   If $\{n\}$ is not a block of $\pi$, then there exist $\alpha\in[a]$ and $\beta\in[b_\alpha]$ such that $B_{\alpha\beta}\cup\{n\}$ is a block of $\pi$. Then, $\pi':=\pi_\alpha'$ is the only partition satisfying $\pi\subseteq\pi'$ and $f(\pi')=\tau'$. Hence (\ref{localsum}) is in this case equal to $p_{\pi,\pi_\alpha'}=\Phi_j(T)$.

   Assume now that $\{n\}$ is a block of $\pi$, i.e., there exists $\ell\in E$ such that $(\{n\},\ell)$ is a labeled block of $\pi$. In this case exactly the partitions $\pi'\in\{\pi_0'(k):k\in E\}\cup\{\pi_1',\ldots,\pi_a'\}$ satisfy $\pi\subseteq\pi'$ and $f(\pi')=\tau'$. Therefore,
   \begin{eqnarray*}
      &   & \hspace{-15mm}
            \pr(f\circ{\cal A}_r^{(n)}=\tau'\,|\,{\cal A}_{r-1}^{(n)}=\pi)
      \ = \ \sum_{\pi'\in f^{-1}(\tau')} p_{\pi,\pi'}
      \ = \ \sum_{k\in E} p_{\pi,\pi_0'(k)}
         + \sum_{\alpha=1}^a p_{\pi,\pi_\alpha'}\nonumber\\
   & = & \sum_{k\in E}\Phi_{j+e_k}(T(k,\ell)) + \sum_{k\in E}\sum_{s=1}^{j_k} \Phi_j(T(k,\ell,s))
   \ = \ \Phi_j(T),
   \end{eqnarray*}
   where the last equality holds by Lemma \ref{consis}, the tensor $T(k,\ell)$ is obtained from $T$ by replacing the vector $t_{k,\ell}=(i_{k,\ell,1},\ldots,i_{k,\ell,j_k})$ by the vector $(i_{k,\ell,1},\ldots,i_{k,\ell,j_k},1)$ and the vector $t_{k,\ell'}=(i_{k,\ell',1},\ldots,i_{k,\ell',j_k})$ by the vector $(i_{k,\ell',1},\ldots,i_{k,\ell',j_k},0)$ for all $\ell'\ne\ell$ and the tensor $T(k,\ell,s)$ is obtained from $T$ by replacing the single entry $i_{k,\ell,s}$ by $i_{k,\ell,s}+1$. Thus,
   \begin{equation} \label{con1}
      \pr(f\circ{\cal A}_r^{(n)}=\tau'\,|\,{\cal A}_{r-1}^{(n)}=\pi)
      \ =\ \left\{
      \begin{array}{cl}
         \Phi_j(T) & \mbox{if $f(\pi)\subseteq\tau'$,}\\
         0 & \mbox{otherwise.}
      \end{array}
      \right.
   \end{equation}
   In particular, (\ref{con1}) depends on $\pi$ only via $\tau:=f(\pi)$. Now, since $\{f\circ{\cal A}_{r-1}^{(n)}=\tau\}\ =\ \bigcup_{\pi\in f^{-1}(\tau)}\{{\cal A}_{r-1}^{(n)}=\pi\}$, an elementary calculation shows that (\ref{con1}) implies
   \[
   \pr(f\circ{\cal A}_r^{(n)}=\tau'\,|\,f\circ{\cal A}_{r-1}^{(n)}=\tau)
   \ =\ \left\{
   \begin{array}{cl}
      \Phi_j(T) & \mbox{if $\tau\subseteq\tau'$,}\\
      0 & \mbox{otherwise.}
   \end{array}
   \right.
   \]
   Since $({\cal A}_r^{(n)})_{r\in\nz_0}$ is Markovian, all these calculations remain valid if in the condition the event $\{f\circ{\cal A}_{r-1}^{(n)}=\tau\}$ is replaced by $\{f\circ{\cal A}_{r-1}^{(n)}=\tau,f\circ{\cal A}_{r-2}^{(n)}=\tau_{r-2},\ldots,f\circ{\cal A}_0^{(n)}=\tau_0\}$. Thus, $(f\circ{\cal A}_r^{(n)})_{r\in\nz_0}$ is a Markov chain with the same transition probabilities (and the same initial state) as $({\cal A}_r^{(m)})_{r\in\nz_0}$.
\end{proof}
Let us now turn to the proofs concerning symmetry and exchangeability issues. For the proof of Lemma \ref{exchangelemma} the following result is needed.
\begin{lemma} \label{technical}
   For all $\pi,\pi'\in{\cal P}_{n,E}$ and all $\sigma\in S_n$ we have $p_{\pi,\pi'}=p_{\sigma(\pi),\sigma(\pi')}$.
\end{lemma}
\begin{proof}
   Two cases are distinguished. Assume first that $\pi\not\subseteq\pi'$. Then there exists some block $B\in\pi'$ that is not a union of blocks of $\pi$. Thus, for all blocks $A_1,\ldots,A_m$ of $\pi$, where $m\le|\pi|$ ($|\pi|$ denotes the number of blocks of $\pi$) we have $\bigcup_{i=1}^m A_i\ne B$. Thus, also $\bigcup_{i=1}^m \sigma(A_i)\ne\sigma(B)$. Since every block of $\sigma(\pi)$ is of the form $\sigma(A)$ for some block $A$ of $\pi$, it follows that $\sigma(\pi)\not\subseteq\sigma(\pi')$. Thus, $p_{\pi,\pi'}=0=p_{\sigma(\pi),\sigma(\pi')}$.

   Assume now that $\pi\subseteq\pi'$. Let $T=((i_{k,\ell,s})_{s\in[j_k]})_{k,\ell\in E}$, where $j_k$ is the number of $k$-blocks of $\pi'$ and $i_{k,\ell,s}$ denotes the number of $\ell$-blocks of $\pi$ merging to the $s$-th $k$-block of $\pi'$. Note that the number $i_\ell$ of $\ell$-blocks of $\pi$ can be recovered from $T$ via $i_\ell=\sum_{k\in E}\sum_{s=1}^{j_k}i_{k,\ell,s}$. Thus, $p_{\pi,\pi'}=\Phi_j(T)$ depends only via $T$ on $\pi$ and $\pi'$.

   Fix $k,\ell\in E$. For $s\in[j_k]$ let $(A_1,\ell),\ldots,(A_{i_{k,\ell,s}},\ell)$ denote the $\ell$-blocks merging to the $s$-th $k$-block $(B,k)$ of $\pi'$. Then $(\sigma(B),k)$ is a $k$-block of $\sigma(\pi')$ and the $\ell$-blocks of $\sigma(\pi)$ merging to $(\sigma(B),k)$ are exactly $(\sigma(A_1),\ell),\ldots,(\sigma(A_{i_{k,\ell,s}}),\ell)$, since it is straightforward to check that $\sigma(A_m)\subseteq\sigma(B)$ and $A\cap\sigma(B)=\emptyset$ for all $\ell$-blocks $(A,\ell)$ of $\sigma(\pi)$, where $A\ne A_m$ for all $m\in[i_{k,\ell,s}]$. Thus, we have exactly $i_{k,\ell,s}$ $\ell$-blocks of $\sigma(\pi)$ merging to this $k$-block of $\sigma(\pi')$. Firstly, this shows that $\sigma(\pi)\subseteq\sigma(\pi')$, since every $k$-block of $\sigma(\pi')$ is of the form $(\sigma(B),k)$ for some $k$-block $(B,k)$ of $\pi'$.

   In the same way as above, let $T'=((a_{k,\ell,s})_{s\in[j_k]})_{k,\ell\in E}$, where $a_{k,\ell,s}$ denotes the number of $\ell$-blocks of $\sigma(\pi)$ merging to the $s$-th $k$-block of $\sigma(\pi')$. 
   The above secondly shows that $(a_{k,\ell,s})_{s\in[j_k]}$ is a permutation of $(i_{k,\ell,s})_{s\in[j_k]}$. By assumption (A) and writing $j=(j_k)_{k\in E}$, we thus obtain $p_{\pi,\pi'}=\Phi_j(T)=\Phi_j(T')=p_{\sigma(\pi),\sigma(\pi')}$.
\end{proof}
Thanks to Lemma \ref{technical} the proof of Lemma \ref{exchangelemma} is now straightforward and works as follows.
\begin{proof}[Proof of Lemma \ref{exchangelemma}]
   We use induction on $r\in\nz_0$. By assumption ${\cal A}_0$ is exchangeable. The induction step from $r-1$ to $r\in\nz$ works as follows. For all $\pi'\in{\cal P}_{n,E}$ and $\sigma\in S_n$,
   \[
   \pr({\cal A}_r=\pi')
   \ =\ \sum_{\pi\in{\cal P}_{n,E}} p_{\pi,\pi'}\pr({\cal A}_{r-1}=\pi)
   \ =\ \sum_{\pi\in{\cal P}_{n,E}} p_{\pi,\pi'}\pr\big({\cal A}_{r-1}=\sigma(\pi)\big)
   \]
   by the induction hypothesis. Lemma \ref{technical} therefore yields
   \[
   \pr({\cal A}_r=\pi')
   \ =\ \sum_{\pi\in{\cal P}_{n,E}}
         p_{\sigma(\pi),\sigma(\pi')}\pr\big({\cal A}_{r-1}=\sigma(\pi)\big).
   \]
   With $\pi$ also $\tau:=\sigma(\pi)$ extents over all labeled partitions in ${\cal P}_{n,E}$. Therefore,
   \[
   \pr({\cal A}_r=\pi')
   \ =\ \sum_{\tau\in{\cal P}_{n,E}}p_{\tau,\sigma(\pi')}\pr({\cal A}_{r-1}=\tau)
   \ =\ \pr\big({\cal A}_r=\sigma(\pi')\big).
   \]
   Thus, ${\cal A}_r$ is exchangeable.
\end{proof}
We now turn to the proofs concerning asymptotic considerations as $N_{\rm min}:=\min_{k\in E}N_k$ tends to infinity.
\begin{proof}[Proof of Lemma \ref{identitylemma}]
   Since $n\ge 2$, the transition matrix $P_N$ has as entries in particular all the mean backward mutation probabilities $\me(N_{k,\ell})/N_\ell$, $k,\ell\in E$, and as well all coalescence probabilities (\ref{coal1}) and (\ref{coal2}). Exploiting the monotonicity (Corollary \ref{monotonecorollary}) it follows that $P_N\to I$ as $N_{\rm min}\to\infty$ if and only if $\frac{N_k}{N_\ell}\me(\nu_{k,\ell,1})=\me(N_{k,\ell})/N_{\ell}\to\delta_{k,\ell}$ (Kronecker symbol) for all $k,\ell\in E$ and if all coalescence probabilities tend to zero as $N_{\rm min}\to\infty$. From
   \[
   c_{k,k}(N_k,N_k)\ =\ \frac{\me(\nu_{k,k,1}^2)-\me(\nu_{k,k,1})}{N_k-1}
   \]
   and the bounds in Lemma \ref{coalbounds} in the appendix for the coalescence probabilities it therefore follows that $P_N\to I$ is equivalent to
   \begin{equation} \label{equivalent1}
      \frac{\me(N_{k,\ell})}{N_\ell}\ \to\ \delta_{k,\ell}
      \quad\mbox{and}\quad \frac{\me(\nu_{k,k,1}^2)}{N_k-1}\ \to\ 0
      \qquad\mbox{for all $k,\ell\in E$.}
   \end{equation}
   Thus, it suffices to verify that (\ref{equivalent1}) is equivalent to
   \begin{equation} \label{equivalent2}
      \me(\nu_{k,k,1})\ \to\ 1\quad\mbox{and}\quad\me(\nu_{k,k,1}\nu_{k,k,2})\ \to\ 1
      \qquad\mbox{for all $k\in E$.}
   \end{equation}
   Assume first that (\ref{equivalent1}) holds. Then, in particular $\me(\nu_{k,k,1})=\me(N_{k,k})/N_k\to 1$. Moreover, since $N_{k,k}/N_k\le 1$, we obtain for the second moment of $N_{k,k}/N_k$ the upper bound
   \begin{equation} \label{upperbound}
      \frac{\me(N_{k,k}^2)}{N_k^2}\ \le\ \frac{\me(N_{k,k})}{N_k}
      \ \to\ 1
   \end{equation}
   and, by Jensen's inequality, the lower bound
   \begin{equation} \label{lowerbound}
      \frac{\me(N_{k,k}^2)}{N_k^2}\ \ge\ \bigg(\frac{\me(N_{k,k})}{N_k}\bigg)^2
      \ \to\ 1.
   \end{equation}
   Thus, $\me(N_{k,k}^2)/(N_k)_2\to 1$ and, hence,
   \[
   \me(\nu_{k,k,1}\nu_{k,k,2})\ =\ \frac{\me(N_{k,k}^2)}{(N_k)_2}
   - \frac{\me(\nu_{k,k,1}^2)}{N_k-1}
   \ \to\ 1-1\ =\ 0.
   \]
   Thus, (\ref{equivalent2}) holds.

   Conversely, assume that (\ref{equivalent2}) holds. Then $\me(N_{k,k})/N_k=\me(\nu_{k,k,1})\to 1$ for all $k\in E$. Thus, for all $\ell\in E$,
   \[
   \me\bigg(\sum_{k\ne\ell}\frac{N_{k,\ell}}{N_\ell}\bigg)
   \ =\ \me\bigg(1-\frac{N_{\ell,\ell}}{N_\ell}\bigg)
   \ =\ 1-\frac{\me(N_{\ell,\ell})}{N_\ell}\ \to\ 1-1\ =\ 0.
   \]
   In particular, $\me(N_{k,\ell})/N_\ell\to 0$ for all $k,\ell\in E$ with $k\ne\ell$. For the second moment of $N_{k,k}/N_k$ the bounds (\ref{upperbound}) and (\ref{lowerbound}) are still valid. Thus, again $\me(N_{k,k}^2)/(N_k)_2\to 1$ and, hence,
   \[
   \frac{\me(\nu_{k,k,1}^2)}{N_k-1}
   \ =\ \frac{\me(N_{k,k}^2)}{(N_k)_2} - \me(\nu_{k,k,1}\nu_{k,k,2})
   \ \to\ 1-1\ =\ 0.
   \]
   Therefore, (\ref{equivalent1}) holds.

   To verify the last statement let $k,\ell\in E$ with $k\ne\ell$ and assume that $P_N\to I$. Then, as seen before, $\me(N_{k,\ell})/N_\ell\to 0$ and, hence,
   \[
   \frac{(N_k)_2}{(N_\ell)_2}\me(\nu_{k,\ell,1}\nu_{k,\ell,2})
   \ =\ \frac{N_k}{(N_\ell)_2}\me\bigg(\nu_{k,\ell,1}\sum_{i=2}^{N_k}\nu_{k,\ell,i}\bigg)
   \ \le\ \frac{N_k}{N_\ell-1}\me(\nu_{k,\ell,1})
   \ =\ \frac{\me(N_{k,\ell})}{N_\ell-1}\ \to\ 0.
   \]
   Thus, $(N_k/N_\ell)^2\me(\nu_{k,\ell,1}\nu_{k,\ell,2})\to 0$ for all $k,\ell\in E$ with $k\ne\ell$.
\end{proof}
Let us now turn to the proofs of the two convergence results, Theorem  \ref{main1} and Theorem \ref{main2}, respectively. Although the two proofs have much in common, we provide these proofs separately, since the proof of Theorem \ref{main1} is less technical and does not rely on any time-scaling.
\begin{proof}[Proof of Theorem \ref{main1}]
   Let $\pi,\pi'\in{\cal P}_{n,E}$. 
   If $\pi\not\subseteq\pi'$, then $p_{\pi,\pi'}^{(N)}=0=:a_{\pi,\pi'}$. Assume now that $\pi\subseteq\pi'$, i.e., each block of $\pi'$ is a union of some blocks of $\pi$ (types of the blocks disregarded here). Let $i_\ell$ and $j_k$ denote the number of $\ell$-blocks of $\pi$ and $k$-blocks of $\pi'$ respectively and let $i_{k,\ell,s}$, $k,\ell\in E$, $s\in[j_k]$, denote the group sizes of $\ell$-blocks of $\pi$ merging to the $s$-th $k$-block of $\pi'$. Furthermore, let $T=(t_{k,\ell})_{k,\ell\in E}$ denote the tensor with entries $t_{k,\ell}:=(i_{k,\ell,s})_{s\in[j_k]}$, $k,\ell\in E$. Then, $p_{\pi,\pi'}^{(N)}=\Phi_j^{(N)}(T)$. If $T\ne\mathbf{1}_j$, then, by assumption, $p_{\pi,\pi'}^{(N)}=\Phi_j^{(N)}(T)\to\phi_j(T)=:a_{\pi,\pi'}$ as $N_{\rm min}\to\infty$. If $T=\mathbf{1}_j$, then $\pi=\pi'$ and
   \[
   p_{\pi,\pi}^{(N)}
   \ =\ 1-\sum_{\pi'\ne\pi}p_{\pi,\pi'}^{(N)}
   \ \to\ 1-\sum_{\pi'\ne\pi}a_{\pi,\pi'}
   \ =:\ a_{\pi,\pi},\qquad N_{\rm min}\to\infty.
   \]
   Thus, the transition matrix $P_N=(p_{\pi,\pi'}^{(N)})_{\pi,\pi'\in{\cal P}_{n,E}}$ of the ancestral process converges to $A$ as $N_{\rm min}\to\infty$, with the matrix $A$ as defined in the statement of the theorem. The convergence of the finite-dimensional distributions of the ancestral process $({\cal A}_r^{(n,N)})_{r\in\nz_0}$ and the convergence in $D_{{\cal P}_{n,E}}(\nz_0)$ to a Markov chain $\Pi^{(n)}=(\Pi_r^{(n)})_{r\in \nz_0}$ with transition matrix $A$ now follows immediately.
\end{proof}
\begin{proof}[Proof of Theorem \ref{main2}]
   Let $\pi,\pi'\in{\cal P}_{n,E}$. 
   If $\pi\not\subseteq\pi'$, then $p_{\pi,\pi'}^{(N)}=0$ and, hence, $p_{\pi,\pi'}^{(N)}/c_N=0=:q_{\pi,\pi'}$. Assume now that $\pi\subseteq\pi'$, i.e., each block of $\pi'$ is a union of some blocks of $\pi$ (types of the blocks disregarded here). Let $i_\ell$ and $j_k$ denote the number of $\ell$-blocks of $\pi$ and $k$-blocks of $\pi'$ respectively and let $i_{k,\ell,s}$, $k,\ell\in E$, $s\in[j_k]$, denote the group sizes of $\ell$-blocks of $\pi$ merging to the $s$-th $k$-block of $\pi'$. Furthermore, let $T=(t_{k,\ell})_{k,\ell\in E}$ denote the tensor with entries $t_{k,\ell}:=(i_{k,\ell,s})_{s\in[j_k]}$, $k,\ell\in E$. Then, $p_{\pi,\pi'}^{(N)}=\Phi_j^{(N)}(T)$. If $T\ne\mathbf{1}_j$, then, by assumption, $p_{\pi,\pi'}^{(N)}/c_N=\Phi_j^{(N)}(T)/c_N\to \phi_j(T)=:q_{\pi,\pi'}$ as $N_{\rm min}\to\infty$. If $T=\mathbf{1}_j$, then $\pi=\pi'$ and
   \[
   \frac{1-p_{\pi,\pi}^{(N)}}{c_N}
   \ =\ \sum_{\pi'\ne\pi}\frac{p_{\pi,\pi'}^{(N)}}{c_N}
   \ \to\ \sum_{\pi'\ne\pi}q_{\pi,\pi'}
   \ =:\ -q_{\pi,\pi},\qquad N_{\rm min}\to\infty.
   \]
   Thus, the transition matrix $P_N=(p_{\pi,\pi'}^{(N)})_{\pi,\pi'\in{\cal P}_{n,E}}$ of the ancestral process $({\cal A}_r^{(n,N)})_{r\in\nz_0}$ satisfies $P_N=I+c_NQ+o(c_N)$ as $N_{\rm min}\to\infty$, with the matrix $Q$ as defined in the statement of the theorem. The convergence of the finite-dimensional distributions of the time-scaled ancestral process $({\cal A}_{\lfloor t/c_N\lfloor}^{(n,N)})_{t\ge 0}$ as $N_{\rm min}\to\infty$ and the convergence in $D_{{\cal P}_{n,E}}([0,\infty))$ to a Markov process $\Pi^{(n)}=(\Pi_t^{(n)})_{t\ge 0}$ with generator $Q$ now follows as in the proof of Theorem 1 of \cite{Moehle2024} by applying Ethier and Kurtz \cite[p.~168, Theorem 2.6]{EthierKurtz1986}.
\end{proof}
We now turn to the proofs concerning integral representations.
\begin{proof}[Proof of Proposition \ref{int1prop}]
   We proceed as in the proof of \cite[Lemma 3.1]{MoehleSagitov2001}. Fix $j=(j_k)_{k\in E}\in\nz_0^E$. The particular diagonal tensor $(t_{k,\ell})_{k,\ell\in E}\in{\cal T}_j$ with diagonal entries $t_{k,k}:=(2,\ldots,2)\in\rz^{j_k}$ for all $k\in E$ is denoted by $\mathbf{2}_j$. If $\phi_j(\mathbf{2}_j)=0$, then the statement holds with $\Lambda_j$ being the zero measure on $\Delta_j$. Assume now that $\phi_j(\textbf{2}_j)>0$. Then, $\me\big(\prod_{k\in E}\prod_{s=1}^{j_k}(\nu_{k,k,s})_2\big)>0$ for all sufficiently large values of $N_{\rm min}:=\min_{k\in E}N_k$. Let $Y_{k,s}$, $k\in E$, $s\in[j_k]$, be random variables with distribution
   \[
   \pr\bigg(\bigcap_{k\in E}\bigcap_{s=1}^{j_k}\{Y_{k,s}=i_{k,s}\}\bigg)
   \ :=\
   \frac{\prod_{k\in E}\prod_{s=1}^{j_k}(i_{k,s})_2}
   {\me\big(\prod_{k\in E}\prod_{s=1}^{j_k}(\nu_{k,k,s})_2\big)}
   \pr\bigg(\bigcap_{k\in E}\bigcap_{s=1}^{j_k}\{\nu_{k,k,s}=i_{k,s}\}\bigg).
   \]
   Note that each $Y_{k,s}$ in general depends on $j=(j_k)_{k\in E}$, but this dependence is suppressed for convenience in our notation. The random variables $Y_{k,s}$ are a size biased modification of the offspring variables $\nu_{k,k,s}$. Note that $Y_{k,k,s}$ takes values in $\{2,\ldots,N_k\}$ almost surely. For all $i_{k,s}\in\nz_0$, $k\in E$, $s\in[j_k]$, we have
   \[
   \me\bigg(
      \prod_{k\in E}\prod_{s=1}^{j_k}Y_{k,s}^{i_{k,s}}
   \bigg)
   \ =\
   \frac{\me\big(
           \prod_{k\in E}\prod_{s=1}^{j_k}
           (\nu_{k,k,s}^{i_{k,s}+2}-\nu_{k,k,s}^{i_{k,s}+1})
        \big)}
   {\me\big(\prod_{k\in E}\prod_{s=1}^{j_k}(\nu_{k,k,s})_2\big)}.
   \]
   Using that $t^i=\sum_{m=0}^i (t)_mS(i,m)$ for all $t\in\rz$ and $i\in\nz_0$, where the $S(.,.)$ denote the Stirling numbers of the second kind, it follows from the assumptions of Theorem \ref{main2} that
   \begin{equation} \label{mom}
      \me\bigg(\prod_{k\in E}\prod_{s=1}^{j_k}
      \bigg(\frac{Y_{k,s}}{N_k}\bigg)^{i_{k,s}}\bigg)
      \ \to\ \frac{\phi_j(T+\mathbf{2}_j)}{\phi_j(\mathbf{2}_j)}
   \end{equation}
   as $N_{\rm min}\to\infty$, where $T=(t_{k,\ell})_{k,\ell\in E}$ denotes the diagonal tensor with diagonal entries $t_{k,k}:=(i_{k,s})_{s\in[j_k]}$, $k\in E$. Note that $(Y_{k,s}/N_k)_{k\in E,s\in[j_k]}$ is concentrated on $\Delta_j$. Since $\Delta_j$ is compact, the convergence of all the moments (\ref{mom}) implies the existence of a probability measure $P_j$ on $\Delta_j$ such that $(Y_{k,s}/N_k)_{k\in E,s\in[j_k]}$ converges in distribution to $P_j$ as $N_{\rm min}\to\infty$. Thus, (\ref{int1}) holds with $Q_j:=\phi_j(\mathbf{2}_j)P_j$. The measure $Q_j$ is uniquely determined, since the limiting moments (\ref{mom}) fully determine $P_j$. For all $j,j'\in\nz_0^E$ with $j\le j'$ it follows from the monotonicity property (\ref{phimon}) that $Q_j(\Delta_j)=\phi_j(\mathbf{2}_j)\ge\phi_{j'}(\mathbf{2}_{j'})=Q_{j'}(\Delta_{j'})$.
\end{proof}

\subsection{Proof of Theorem \ref{main3}} \label{proofmain3}
   We extend the proof of the `only if' part of Schweinsberg \cite[Theorem 2]{Schweinsberg2000} to the multi-type setting. Let $\Pi=(\Pi_t)_{t\ge 0}$ denote the multi-type exchangeable coalescent with rates $\phi_j(T)$, $j=(j_k)_{k\in E}\in\nz_0^E$, $T\in{\cal T}_j$. We furthermore denote with $\Pi_t^{(n)}:=\varrho_n\circ\Pi_t$ the restriction of $\Pi_t$ to $[n]$ and by $\Pi^{(n)}:=(\Pi_t^{(n)})_{t\ge 0}$ the corresponding $n$-coalescent. Define the stopping time
   \[
   S\ :=\ \inf\{t>0:\mbox{$1$ and $2$ are in the same block of $\Pi_t$}\}
   \]
   and the events
   \[
   E_n\ :=\ \{1,\ldots,n\mbox{ are in distinct blocks of $\Pi_{S-}$}\},
   \quad n\in\nz.
   \]
   As in \cite{Schweinsberg2000} it follows that $\pr(E_n)>0$. For $n\in\nz$ let $\Theta_n$ be a random labeled partition of $[n]$ whose distribution is the conditional distribution of $\Pi_S$ given $E_n$. We claim that there exists a random labeled partition $\Theta$ of $\nz$ such that $\varrho_n\circ\Theta$ has the same distribution as $\Theta_n$ for all $n\in\nz$. To prove this, let us first verify that $\varrho_{n,m}\circ\Theta_n$ has the same distribution as $\Theta_m$ for all $m,n\in\nz$ with $m<n$. Fix $m<n$. Let $\theta\in{\cal P}_{m,E}$ such that $1$ and $2$ are in the same block of $\theta$. For $k\in E$ let $j_k$ denote the number of $k$-blocks of $\theta$ and $i_{k,s}\in\nz$, $s\in[j_k]$, the corresponding block sizes. Define $\phi_1(\mathbf{2}):=\sum_{k\in E}\phi_{e_k}(\mathbf{2}_k)$, where $\mathbf{2}_k\in{\cal T}_{e_k}$ is the tensor with $i_{k,1}=2$. Note that $\phi_1(\mathbf{2})<\infty$. Furthermore, let $T$ denote the diagonal tensor with diagonal entries $t_{k,k}:=(i_{k,s})_{s\in[j_k]}$, $k\in E$. If $\Pi_t^{(m)}$ has $m$ singleton blocks, then these blocks merge to the blocks of $\theta$ at the rate $\phi_j(T)$. The total rate of all coalescence events involving $1$ and $2$ is $\phi_1(\mathbf{2})$. Thus,
   \begin{equation} \label{thetadist}
      \pr(\Theta_m=\theta)\ =\ \frac{\phi_j(T)}{\phi_1(\mathbf{2})}.
   \end{equation}
   In the same way, if $\Pi^{(n)}_t$ consists of $n$ singletons, then the total rate of all coalescence events of $[n]$, such that their restriction to $[m]$ merges to the blocks of $\theta$, is also $\phi_j(T)$. Therefore, $\pr(\varrho_{n,m}\circ\Theta_n=\theta)=\phi_j(T)/\phi_1(\mathbf{2})$. Thus, $\varrho_{n,m}\circ\Theta_n$ has the same distribution as $\Theta_m$. A standard application of the Daniell--Kolmogorov theorem yields the existence of $\Theta$ as the projective limit of the sequence $(\Theta_n)_{n\in\nz}$.

   Now, denote by $\Theta'$ the restriction of $\Theta$ to $\{3,4,\ldots\}$. Since $\Pi$ is exchangeable, so is $\Theta'$. Since exchangeability in the single-type and multi-type case coincide, copying the proof of Aldous \cite[Section 11, p.~84 ff.]{Aldous1985}, we obtain the existence of the limiting frequencies of the blocks of $\Theta'$, equipped with a (random) label, the label of the respective block. Write $(P_1,K_1),(P_2,K_2),\ldots$ for the pairs of limiting frequencies and their labels, where the $P_i$ are ordered decreasingly and $K_i$ is the label of $P_i$, $i\in\nz$.
   We have $P_n=0$ if $\Theta'$ has fewer than $n$ blocks with non-zero limiting frequencies. Let $P_0:=1-\sum_{j=1}^\infty P_j$. Note that the blocks of $\Theta$ also have limiting relative frequencies that coincide with those of $\Theta'$. Write $B_1,B_2,\ldots$ for the blocks of $\Theta$, such that block $B_i$ has limiting frequency $P_i$ on $\{P_i>0\}$ and blocks with the same limiting relative frequencies are ordered at random, independently of $\Theta$. The block $B_i$ is undefined on $\{P_i=0\}$.

   Let $(x_\ell)_{\ell\in E}$ be some sequence of distinct real numbers in $[0,1)$. Define a sequence of random variables $Z_1,Z_2,\ldots$ via $Z_m:=i+x_\ell$ on the event $\{m\in B_i,K_i=\ell\}$ and $Z_m:=0$ on the complement of $\bigcup_{i\in\nz}\bigcup_{\ell\in E}\{m\in B_i,K_i=\ell\}$, i.e., if the block containing $m$ has limiting frequency $0$. Furthermore, let ${\cal G}$ denote the $\sigma$-algebra generated by $(P,K)$, where
   $P:=(P_i)_{i\in\nz}$ and $K:=(K_i)_{i\in\nz}$, and let $(x,x):=\sum_{i=1}^\infty x_i^2$ for $x=(x_i)_{i\in\nz}\in\Delta$.
\begin{lemma} \label{local}
  For all $m\in\{3,4,\ldots\}$, $i\in\nz$ and $\ell\in E$ we have $\pr(Z_m=i+x_\ell\,|\,{\cal G})=1_{\{K_i=\ell\}}P_i$ almost surely and $\pr(Z_m=0\,|\,{\cal G})=P_0$ almost surely. Furthermore, $\pr(Z_1=i+x_\ell\,|\,{\cal G})=1_{\{K_i=\ell\}}P_i^2/(P,P)$ almost surely on $\{P_1>0\}$. Moreover, $Z_1,Z_3,Z_4,\ldots$ are conditionally independent given ${\cal G}$.
\end{lemma}
\begin{proof}
   Since $\Theta'$ is exchangeable, also $Z_3,Z_4,\ldots$ are exchangeable. The $\sigma$-field generated by its limiting empirical distribution is ${\cal G}$. Therefore, by \cite[Lemma 38]{Schweinsberg2000}, for all $m\in\{3,4,\ldots\}$, $i\in\nz$ and $\ell\in E$,
   \begin{eqnarray*}
      \pr(Z_m=i+x_\ell\,|\,{\cal G})
      & = & \lim_{n\to\infty}\frac{1}{n}\sum_{m=1}^n 1_{\{Z_m=i+x_\ell\}}
      \ = \ \lim_{n\to\infty}
            \frac{1}{n}\sum_{m=1}^n 1_{\{m\in B_i,K_i=\ell\}}\\
      & = & 1_{\{K_i=\ell\}}\lim_{n\to\infty}\frac{1}{n}\sum_{m=1}^n
            1_{\{m\in B_i\}}
      \ = \ 1_{\{K_i=\ell\}}P_i\quad\mbox{almost surely.}
   \end{eqnarray*}
%
%
   Moreover,
   \begin{eqnarray*}
      \pr(Z_m=0\,|\,{\cal G})
      & = & 1-\sum_{i=1}^\infty\sum_{\ell\in E}\pr(Z_m=i+x_\ell\,|\,{\cal G})\\
      & = & 1-\sum_{i=1}^\infty\sum_{\ell\in E}P_i1_{\{K_i=\ell\}}
   \ =\ 1-\sum_{i=1}^\infty P_i\ =\ P_0\quad\mbox{almost surely.}
   \end{eqnarray*}
   Since $\Pi$ is exchangeable, the sequences $(Z_1,Z_3,Z_4,\ldots)$ and $(Z_1,Z_{\sigma(3)},Z_{\sigma(4)},\ldots)$ have the same distribution for all finite permutations $\sigma$ of $\{3,4,\ldots\}$. Be Lemma 39 in Appendix A of \cite{Schweinsberg2000}, $Z_1,Z_3,Z_4,\ldots$ are conditionally independent given ${\cal G}$. It remains to show that $\pr(Z_1=i+x_\ell\,|\,{\cal G})=1_{\{K_i=\ell\}}P_i^2/(P,P)$ almost surely on $\{P_1>0\}$. Fix $n,k\in\nz$ with $4\le k<n$ and define the stopping time
   \[
   S_k\ :=\ \inf\{t>0:\mbox{$k-1$ and $k$ are in the same block of $\Pi_t$}\}
   \]
   and the event
   \[
   E_{n,k}\ :=\ \{\mbox{$1,\ldots,n$ are in distinct blocks of $\Pi_{{S_k}-}$}\}.
   \]
   Furthermore, for all $\ell,m\in E$ let
   \begin{eqnarray*}
      \theta_{k,1,m,\ell} & := & \{(\{1,2\},m),(\{3,\ldots,k\},\ell)\}\ \in\ {\cal P}_{k,E},\\
      \theta_{k,2,\ell,m} & := & \{(\{1,\ldots,k-2\},\ell),(\{k-1,k\},m)\}\ \in\ {\cal P}_{k,E},\\
      M_{1,\ell} & := & \{\theta_{k,1,m,\ell}:m\in E\},\\
      M_{2,\ell} & := & \{\theta_{k,2,\ell,m}:m\in E\}.
   \end{eqnarray*}
   Now, let $\Pi^{(k+1,n)}$ denote the restriction of $\Pi$ to $\{k+1,\ldots,n\}$ and let $\pi\in{\cal P}_{\{k+1,\ldots,n\},E}$. The exchangeability of $\Pi$ implies
   \[
   \pr(\Pi_S^{(k)}=\theta_{k,1,m,\ell},\Pi_S^{(k+1,n)}=\pi)
   \ =\ \pr(\Pi_{S_k}^{(k)}=\theta_{k,2,\ell,m},\Pi_{S_k}^{(k+1,n)}=\pi)
   \]
   and therefore, by summation over all $m\in E$,
   \[
   \pr(\Pi_S^{(k)}\in M_{1,\ell},\Pi_S^{(k+1,n)}=\pi)
   \ =\ \pr(\Pi_{S_k}^{(k)}\in M_{2,\ell},\Pi_{S_k}^{(k+1,n)}=\pi).
   \]
   Since $S=S_k$ on $\{\Pi_{S_k}^{(k)}\in M_{2,\ell}\}\cap E_{n,k}$ and
   $\{\Pi_S^{(k)}\in M_{2,\ell}\}\cap E_n$, we obtain
   \[
   \pr(\{\Pi_S^{(k)}\in M_{1,\ell}\}\cap E_n\cap\{\Pi_S^{(k+1,n)}=\pi\})
   \ =\ \pr(\{\Pi_S^{(k)}\in M_{2,\ell}\}\cap E_n\cap \{\Pi_{S_k}^{(k+1,n)}=\pi\}).
   \]
   Therefore,
   \[
   \pr(\Pi_S^{(k)}\in M_{1,\ell},\Pi_S^{(k+1,n)}=\pi\,|\,E_n)
   \ =\ \pr(\Pi_S^{(k)}\in M_{2,\ell},\Pi_{S_k}^{(k+1,n)}=\pi\,|\,E_n).
   \]
   Let $\Theta_{k+1,n}$ denote the restriction of $\Theta$ to $\{k+1,\ldots,n\}$ and ${\cal G}_{k+1,n}$ its generated $\sigma$-field.
   Since $\Theta_n$ is by definition the conditional distribution of $\Pi_S$ given $E_n$, it follows that
   \[
   \pr(\Theta_k\in M_{1,\ell},\Theta_{k+1,n}=\pi)\ =\ \pr(\Theta_k\in M_{2,\ell},\Theta_{k+1,n}=\pi).
   \]
   Because this holds for arbitrary $\pi\in{\cal P}_{\{k+1,\ldots,n\},E}$ we obtain
   \[
   \pr(\Theta_k\in M_{1,\ell}\,|\,{\cal G}_{k+1,n})
   \ =\ \pr(\Theta_k\in M_{2,\ell}\,|\,{\cal G}_{k+1,n})
   \quad\mbox{almost surely.}
   \]
   Now, let ${\cal G}_{k+1}$ denote the $\sigma$-field generated by the restriction of $\Theta'$ to $\{k+1,k+2,\ldots\}$. Because $\bigcup_{n\in\nz}{\cal G}_{k+1,n}={\cal G}_{k+1}$, a standard martingale argument yields
   \[
   \pr(\Theta_k\in M_{1,\ell}\,|\,{\cal G}_{k+1})
   \ =\ \pr(\Theta_k\in M_{2,\ell}\,|\,{\cal G}_{k+1})\quad\mbox{almost surely.}
   \]
   The limiting relative frequencies and (random) colors of blocks can be recovered from the restriction of $\Theta'$ to $\{k+1,k+2,\ldots\}$, so $(P_i,K_i)_{i\in\nz}$ is ${\cal G}_{k+1}$-measurable. Therefore, ${\cal G}\subseteq{\cal G}_{k+1}$ and by the tower property of conditional expectations
   \begin{equation} \label{help}
      \pr(\Theta_k\in M_{1,\ell}\,|\,{\cal G})
      \ =\ \pr(\Theta_k\in M_{2,\ell}\,|\,{\cal G})\quad\mbox{almost surely.}
   \end{equation}
   In the following it is shown that (\ref{help}) leads to the statement of the lemma: By the definition of $Z_1,Z_2,\ldots$ it follows by conditional independence that
   \begin{eqnarray*}
      \pr(\Theta_k\in M_{1,\ell}\,|\,{\cal G})
      & = & \sum_{i=1}^\infty
            \pr(Z_3=\cdots=Z_k=i+x_\ell,Z_1=Z_2\ne i+x_\ell\,|\,{\cal G})\\
      & = & \sum_{i=1}^\infty \pr(Z_3=\cdots=Z_k=i+x_\ell\,|\,{\cal G})
             \pr(Z_1\ne i+x_\ell\,|\,{\cal G})\\
      & = & \sum_{i=1}^\infty P_i^{k-2}1_{\{K_i=\ell\}}(1-Q_{i,\ell})
   \end{eqnarray*}
   almost surely, where $Q_{i,\ell}:=\pr(Z_1=i+x_\ell\,|\,{\cal G})$ for all $i\in\nz$ and $\ell\in E$. Similarly,
   \begin{eqnarray*}
      &   & \hspace{-10mm}\pr(\Theta_k\in M_{2,\ell}\,|\,{\cal G})\\
      & = & \sum_{i\ne j}\sum_{m\in E}
            \pr(Z_1=Z_2=i+x_\ell,Z_3=\cdots=Z_{k-2}=i+x_\ell,Z_{k-1}=Z_k=j+x_m|{\cal G})\\
      & = & \sum_{i\ne j}\sum_{m\in E}
            \pr(Z_1=i+x_\ell\,|\,{\cal G})
            \pr(Z_3=\cdots=Z_{k-2}=i+x_\ell\,|\,{\cal G})
            \pr(Z_{k-1}=Z_k=j+x_m\,|\,{\cal G})\\
      & = & \sum_{i\ne j}\sum_{m\in E}
            Q_{i,\ell}P_i^{k-4}1_{\{K_i=\ell\}}
            P_j^21_{\{K_j=m\}}\\
      & = & \sum_{i=1}^\infty Q_{i,\ell}P_i^{k-4}1_{\{K_i=\ell\}}
            \sum_{j\ne i}P_j^2\sum_{m\in E}1_{\{K_j=m\}}\\
      & = & \sum_{i=1}^\infty Q_{i,\ell}P_i^{k-4}1_{\{K_i=\ell\}}
            \sum_{j\ne i}P_j^2\quad\mbox{almost surely}.
   \end{eqnarray*}
   It follows from (\ref{help}) that
   \begin{eqnarray*}
      0
      & = & \pr(\Theta_k\in M_{1,\ell}\,|\,{\cal G})
            - \pr(\Theta_k\in M_{2,\ell}\,|\,{\cal G})\\
      & = & \sum_{i=1}^\infty \bigg(
               P_i^{k-2}1_{\{K_i=\ell\}}(1-Q_{i,\ell})
               -Q_{i,\ell}P_i^{k-4}1_{\{K_i=\ell\}\sum_{j\ne i}P_j^2}
            \bigg)\\
      & = & \sum_{i=1}^\infty P_i^{k-4}1_{\{K_i=\ell\}}
            \bigg(
               P_i^2-P_i^2Q_{i,\ell}-Q_{i,\ell}\sum_{j\ne i}P_j^2
            \bigg)\\
      & = & \sum_{i=1}^\infty P_i^{k-4}1_{\{K_i=\ell\}}(P_i^2-Q_{i,\ell}(P,P))
      \quad\mbox{almost surely.}
   \end{eqnarray*}
   From \cite[Lemma 20]{Schweinsberg2000} we obtain $Q_{i,\ell}=P_i^2/(P,P)$ on
   $\{P_i>0,K_i=\ell\}$. Thus,
   \begin{equation} \label{nearly}
      Q_{i,\ell}\ =\ 1_{\{K_i=\ell\}}\frac{P_i^2}{(P,P)}
   \end{equation}
   on $\{P_i>0\}$. On $\{P_1>0\}$ we have
   $\sum_{j\in\nz}\sum_{\ell\in E}Q_{j,\ell}=1$ almost surely by (\ref{nearly}). Thus, $Q_{i,\ell}=0$ almost surely on $\{P_1>0,P_i=0\}$. It follows from the paragraph after (39) of \cite{Schweinsberg2000} that (\ref{nearly}) holds also on $\{P_1>0\}$ for all $i\in\nz$.
\end{proof}
   Thanks to Lemma \ref{local}, the proof of Theorem \ref{main3} is now completed as follows. Let $j=(j_k)_{k\in E}\in\nz_0^E$ and $T$ be a diagonal tensor with diagonal entries $t_{k,k}:=(i_{k,s})_{s\in[j_k]}\in\{2,3,\ldots\}^{j_k}$. Define $|T|:=\sum_{k\in E}\sum_{s=1}^{j_k}i_{k,s}$. Let $\theta\in{\cal P}_{|T|,E}$ with $j_k$ $k$-blocks $B_{k,1},\ldots,B_{k,j_k}$ of sizes $i_{k,1},\ldots,i_{k,j_k}$ respectively, and such that $1$ and $2$ are in the same block of $\theta$. It is already shown in (\ref{thetadist}) that
   \[
   \phi_j(T)\ =\ \phi_1(\mathbf{2})\pr(\Theta_{|T|}=\theta).
   \]
   Lemma 40 of \cite{Schweinsberg2000} carries over to the multi-type setting and implies that almost surely every block of $\Theta'$ having limiting relative frequency zero is a singleton. So if $i,j\ge 3$, $i$ and $j$ are in the same block if and only if $Z_i=Z_j\ne 0$ almost surely. On $\{P_1>0\}$, Lemma \ref{local} implies $Z_1=Z_2>0$ almost surely, so $i,j\in\nz$ are in the same block if and only if $Z_i=Z_j\ne 0$. This shows in particular that, on $\{P_1>0\}$, $K_i$ is the label of $P_i$ also of $\Theta$. Therefore, on $\{P_1>0\}$, the event $\{\Theta_{|T|}=\theta\}$ coincides, up to a null set, with the event that there exist pairwise distinct $m_{k,s}$, $k\in E$, $s\in[j_k]$, satisfying
   \[
   Z_m\ =\ m_{k,s}+x_k\quad\mbox{for all $k\in E$, $s\in[j_k]$ and $m\in B_{k,s}$}.
   \]
   Let $\ell_0\in E$ with $1\in D_{1,\ell_0}$. Then, by Lemma \ref{local}, for all $k\ne\ell_0$ and all $s\in[j_k]$,
   \[
   \pr(Z_m=m_{k,s}+x_k\mbox{ for all }m\in B_{k,s}\,|\,{\cal G})
   \ =\ P_{m_{k,s}}^{i_{k,s}}1_{\{K_{m_{k,s}}=k\}}
   \]
   and, furthermore, for all $s\in\{2,\ldots,j_{\ell_0}\}$,
   \[
   \pr(Z_m=m_{\ell_0,s}+x_{\ell_0}\mbox{ for all }m\in B_{\ell_0,s}\,|\,{\cal G})
   \ =\ P_{m_{\ell_0},s}^{i_{\ell_0,s}}1_{\{K_{m_{\ell_0,s}}=k\}}
   \]
   and
   \begin{eqnarray*}
      &   & \hspace{-15mm}
            \pr(Z_m=m_{\ell_0,1}+x_{\ell_0}\mbox{ for all }m\in B_{\ell_0,1}
            \,|\,{\cal G})\\
      & = & P_{m_{\ell_0},1}^{i_{\ell_0,1}}
            1_{\{K_{m_{\ell_0,s}}=\ell_0\}}^{i_{\ell_0,1}-2}
            \pr(Z_1=m_{\ell_0,1}+x_{\ell_0}\,|\,{\cal G})\\
      & = & P_{m_{\ell_0,1}}^{i_{\ell_0,1}}
            1_{\{K_{m_{\ell_0,s}}=\ell_0\}}^{i_{\ell_0,1}-1}/(P,P)
      \ = \ P_{m_{\ell_0,1}}^{i_{\ell_0,1}}
            1_{\{K_{m_{\ell_0,s}}=\ell_0\}}/(P,P)
   \end{eqnarray*}
   almost surely on $\{P_1>0\}$, where the last equality holds since $i_{\ell_0,1}\ge 2$. Summation over all possible values of $m_{k,s}$, $k\in E$, $s\in[j_k]$, yields
   \[
   \pr(\Theta_{|T|}=\theta|{\cal G})
   \ =\ \frac{1}{(P,P)}\sum_{m_{k,s}}\prod_{k\in E}\prod_{s=1}^{j_k}
   P_{m_{k,s}}^{i_{k,s}}1_{\{k\}}(K_{m_{k,s}})
   \]
   almost surely on $\{P_1>0\}$. On $\{P_1=0\}$, Lemma \ref{local} implies that $\pr(Z_m=0\,|\,{\cal G})=P_0=1$ almost surely and, therefore, $Z_m=0$ almost surely for all $m\in\{3,4,\ldots\}$, so $\Theta'$ has only singletons almost surely. The exchangeability of $\Pi$ implies that, conditional on the event that $\Theta'$ has only singletons, the probability that $1,2$ and $q$ are in the same block of $\Theta$ is the same for all $q\in\{3,4,\ldots\}$ and must hence be zero. It follows that, on $\{P_1=0\}$, $\Theta_{|T|}$ takes values in the set of partitions of the form $\pi_{k_1,\ldots,k_{|T|-1}}:=((\{1,2\},\ell_1),(\{3\},\ell_2)\ldots,(\{|T|\},\ell_{|T|-1}))$ for some $\ell_1,\ldots,\ell_{|T|-1}\in E$ almost surely. In other words, $\Theta_{|T|}=((\{1,2\},L_1),(\{3\},L_2),\ldots,(\{|T|\},L_{|T|-1}))$ almost surely on $\{P_1=0\}$ for some $E$-valued random variables $L_1,\ldots,L_{|T|-1}$. Since all the blocks of $\theta\in{\cal P}_{|T|,E}$ have size at least $2$, it follows that
   $\pr(\Theta_{|T|}=\theta|{\cal G})=1_{\{L_1=k\}}$ almost surely on $\{P_1=0\}$ if $j=e_k$ and $i_{k,1}=2$ for some $k\in E$ and $\pr(\Theta_{|T|}=\theta|{\cal G})=0$ almost surely on $\{P_1=0\}$ otherwise. Combining this, we obtain
   \begin{eqnarray} \label{localdist}
      \pr(\Theta_{|T|}=\theta|{\cal G})
      & = & 1_{\{P_1=0\}}\sum_{k\in E}
            1_{\{L_1=k\}}
            1_{\{j=\mathbf{e}_k,i_{k,1}=2\}}\nonumber\\
      &   & \hspace{1cm} + 1_{\{P_1>0\}}
              \frac{1}{(P,P)}\sum_{m_{k,s}}\prod_{k\in E}\prod_{s=1}^{j_k}
              P_{m_{k,s}}^{i_{k,s}}1_{\{k\}}(K_{m_{k,s}})
   \end{eqnarray}
   almost surely. Analogous to Aldous \cite{Aldous1985} it follows that $(P,K)$ takes values in $\Delta\times E^\nz$ almost surely. Let $Q$ denote the distribution of $(P,K)$ restricted to $(\Delta\setminus\{0\})\times E^\nz$. Furthermore, define
   $\alpha_k:=\pr(P_1=0,L_1=k)$ for all $k\in E$.
   Taking the expectation in (\ref{localdist}) yields
   \[
   \pr(\Theta_{|T|}=\theta)
   \ =\
   \sum_{k\in E}\alpha_k1_{\{j=\mathbf{e}_k,i_{k,1}=2\}}
            + \int_{(\Delta\setminus\{\mathbf{0}\})\times E^\nz}
            \sum_{m_{k,s}}\prod_{k\in E}\prod_{s=1}^{j_k}
              x_{m_{k,s}}^{i_{k,s}}1_{\{k\}}(y_{m_{k,s}})
              \,\frac{Q({\rm d}(x,y))}{(x,x)}.
   \]
   Since $\phi_j(T)=c\pr(\Theta_{|T|}=\theta)$ with $c:=\phi_j(\mathbf{2})$, the statement of Theorem \ref{main3} follows with $\Xi:=cQ$ and  $a_k:=c\alpha_k$ for all $k\in E$.


\vfill\eject

\subsection{Appendix} \label{appendix}
This appendix collects some basic definitions, illustrations and results used in the article.
\subsubsection{Some formulas concerning the number of offspring}
We provide some straightforward formulas and bounds for certain (joint) moments of numbers of offspring.
\begin{lemma}
   Under (A),
   \begin{equation} \label{mean1}
      \me(\nu_{k,\ell,i})\ =\
      \frac{\me(N_{k,\ell})}{N_k},
      \qquad k,\ell\in E,i\in[N_k],
   \end{equation}
   and
   \begin{equation} \label{mean2}
      \me(\nu_{k,\ell,1}\nu_{k,\ell,2})
      \ =\ \frac{{\rm Var}(N_{k,\ell})}{(N_k)_2}
      + \bigg(\frac{\me(N_{k,\ell})}{N_k}\bigg)^2
      - \frac{{\rm Var}(\nu_{k,\ell,1})}{N_k-1},\qquad k,\ell\in E,
   \end{equation}
   provided that $N_k>1$.
\end{lemma}
\begin{proof}
   Fix $k,\ell\in E$. Eq.~(\ref{mean1}) follows from $\me(N_{k,\ell})=\me(\sum_{i=1}^{N_k}\nu_{k,\ell,i})=N_k\me(\nu_{k,\ell,1})$. Moreover, $\me(N_{k,\ell}^2)=\me(\sum_{i=1}^{N_k}\nu_{k,\ell,i}\sum_{j=1}^{N_k}\nu_{k,\ell,j})=N_k\me((\nu_{k,\ell,1})^2)+(N_k)_2\me(\nu_{k,\ell,1}\nu_{k,\ell,2})$. Assume now that $N_k>1$. Solving for $\me(\nu_{k,\ell,1}\nu_{k,\ell,2})$ yields
   \begin{eqnarray*}
      &   & \hspace{-11mm}\me(\nu_{k,\ell,1}\nu_{k,\ell,2})
      \ = \ \frac{\me(N_{k,\ell}^2)-N_k\me(\nu_{k,\ell,1}^2)}
      {(N_k)_2}
      \ = \ \frac{\me(N_{k,\ell}^2)-N_k(\me(\nu_{k,\ell,1}))^2-N_k{\rm Var}(\nu_{k,\ell,1})}{(N_k)_2}\\
      & = & \frac{\me(N_{k,\ell}^2) -N_k(\frac{\me(N_{k,\ell})}{N_k})^2}{(N_k)_2}
             -\frac{{\rm Var}(\nu_{k,\ell,1})}{N_k-1}
      \ = \ \frac{{\rm Var}(N_{k,\ell})}{(N_k)_2}+\bigg(\frac{\me(N_{k,\ell})}{N_k}\bigg)^2
               - \frac{{\rm Var}(\nu_{k,\ell,1})}{N_k-1},
   \end{eqnarray*}
   which is (\ref{mean2}).
\end{proof}
The following lemma provides upper bounds for the coalescence probabilities defined in (\ref{coal1}) and (\ref{coal2}), respectively.
\begin{lemma} \label{coalbounds}
   For all $k,\ell\in E$,
   \[
   c_{k,\ell}(N_k,N_\ell)
   \ \le\ \frac{\me((N_{k,\ell})_2)}{(N_\ell)_2}
   \ \le\ \frac{\me(N_{k,\ell})}{N_\ell}
   \]
   and, for all $k,\ell_1,\ell_2\in E$ with $\ell_1\ne\ell_2$,
   \[
   c_{k,\ell_1,\ell_2}(N_k,N_{\ell_1},N_{\ell_2})
   \ \le\ \frac{\me(N_{k,\ell_1}N_{k,\ell_2})}{N_{\ell_1}N_{\ell_2}}
   \ \le\ \min\bigg(
   \frac{\me(N_{k,\ell_1})}{N_{\ell_1}},\frac{\me(N_{k,\ell_2})}{N_{\ell_2}}
   \bigg).
   \]
\end{lemma}
\begin{proof}
   For all $k,\ell\in E$, by (\ref{coal1}),
   \begin{eqnarray*}
      c_{k,\ell}(N_k,N_\ell)
      & = & \frac{1}{(N_\ell)_2}
            \me\bigg(
               \sum_{i=1}^{N_k}\nu_{k,\ell,i}(\nu_{k,\ell,i}-1)
            \bigg)
      \ \le \ \frac{1}{(N_\ell)_2}
            \me\bigg(
               \sum_{i=1}^{N_k}\nu_{k,\ell,i}(N_{k,\ell}-1)
            \bigg)\\
      & = & \frac{\me((N_{k,\ell})_2)}{(N_\ell)_2}
      \ \le \ \frac{\me(N_{k,\ell}(N_\ell-1))}{(N_\ell)_2}
      \ =\ \frac{\me(N_{k,\ell})}{N_\ell}.
   \end{eqnarray*}
   Similarly, for all $k,\ell_1,\ell_2\in E$ with $\ell_1\ne\ell_2$, by (\ref{coal2}),
   \begin{eqnarray*}
      c_{k,\ell_1,\ell_2}(N_k,N_{\ell_1},N_{\ell_2})
      & = & \frac{1}{N_{\ell_1}N_{\ell_2}}
         \me\bigg(\sum_{i=1}^{N_k}\nu_{k,\ell_1,i}\nu_{k,\ell_2,i}\bigg)\\
      & \le & \frac{1}{N_{\ell_1}N_{\ell_2}}
           \me\bigg(\sum_{i=1}^{N_k}\nu_{k,\ell_1,i}N_{k,\ell_2}\bigg)
      \ = \ \frac{\me(N_{k,\ell_1}N_{k,\ell_2})}{N_{\ell_1}N_{\ell_2}}.
   \end{eqnarray*}
   The last inequality follows from $N_{k,\ell_1}/N_{\ell_1}\le 1$ and $N_{k,\ell_2}/N_{\ell_2}\le 1$.
\end{proof}

\subsubsection{Block structure of the backward transition matrix} \label{structure}
If the states $\pi$ of ${\cal P}_{n,E}$ are ordered with respect to the number $|\pi|$ of blocks of $\pi$, then the transition matrix $P$ of the ancestral process $({\cal A}_r^{(n,N)})_{r\in\nz_0}$ is a left lower triangular block matrix of the form $P=(P_{i,j})_{1\le i,j\le n}$, where each $P_{i,j}$ is a $(d^iS(n,i)\times d^jS(n,j))$-matrix with $P_{i,j}=0$ for $i<j$, where $S(.,.)$ denote the Stirling numbers of the second kind. For all $i\ge j$ the block matrix $P_{i,j}$ contains all the transition probabilities of transitions from states $\pi$ having $i$ blocks to states $\pi'$ having $j$ blocks. In the following this block structure is provided in detail for sample sizes $n=1$ and $n=2$. It is also assumed that $d:=|E|<\infty$ for simplicity.
For $n=1$ the state space ${\cal P}_{1,E}$ of the ancestral process has size $|{\cal P}_{1,E}|=d$ and consists of all labeled partitions of the form $\pi_k:=\{(\{1\},k)\}$, $k\in E$. By (\ref{rtrans2}), $p_{\pi_k,\pi_\ell}=\me(N_{\ell,k})/N_k$, $k,\ell\in E$. Thus, for $n=1$ the transition matrix $P=(p_{\pi,\pi'})_{\pi,\pi'\in{\cal P}_{1,E}}$ of the ancestral process coincides with the mean backward mutation matrix $M:=(m_{k,\ell})_{k,\ell\in E}$, where $m_{k,\ell}:=\me(N_{\ell,k})/N_k$ is the mean backward mutation probability, i.e. the mean proportion of the individuals in subpopulation $k$ after the reproduction step, who were born in subpopulation $\ell$.

For $n=2$,
$P=\left(
    \begin{array}{cc}
       M & 0\\
       C & D\\
    \end{array}
\right)$
where $M=(m_{k,\ell})_{k,\ell\in E}$ is again the mean backward mutation matrix, $C$ contains the coalescence probabilities and $D$ the remaining transition probabilities $p_{\pi,\pi'}$ for all states $\pi$ and $\pi'$ having two blocks, which can be calculated using (\ref{rtrans2}). For illustration, for $E:=\{1,2\}$, and the state space ordered by
\[
\pi_1\ :=\ \{(\{1,2\},1)\},\qquad\pi_2\ :=\ \{(\{1,2\},2)\},
\]
and
\[
\pi_3\ :=\ \{(\{1\},1),(\{2\},1)\},\qquad\pi_4\ :=\ \{(\{1\},1),(\{2\},2)\},
\]
\[
\pi_5\ :=\ \{(\{1\},2),(\{2\},1)\},\qquad\pi_6\ :=\ \{(\{1\},2),(\{2\},2)\},
\]
the transition matrix $P=(p_{\pi_i,\pi_j})_{1\le i,j\le 6}$ of the ancestral process is provided in Figure \ref{figure3}.
\begin{landscape}
\begin{figure}[p] 
   \caption[]{Transition matrix $P$ of the ancestral process for sample size $n=2$ and $E=\{1,2\}$}
   \centering
   \[
   P\ =\
   \left(
   \begin{array}{cc|cccc}
      \me(\nu_{1,1,1}) & \frac{N_2\me(\nu_{2,1,1})}{N_1} & 0 & 0 & 0 & 0\\
      \frac{N_1\me(\nu_{1,2,1})}{N_2} & \me(\nu_{2,2,1}) & 0 & 0 & 0 & 0\\
      \hline
      \frac{\me((\nu_{1,1,1})_2)}{N_1-1} & \frac{N_2\me((\nu_{2,1,1})_2)}{(N_1)_2} & \me(\nu_{1,1,1}\nu_{1,1,2}) & \frac{N_2\me(\nu_{1,1,1}\nu_{2,1,1})}{N_1-1} & \frac{N_2\me(\nu_{1,1,1}\nu_{2,1,1})}{N_1-1} & \frac{(N_2)_2\me(\nu_{2,1,1}\nu_{2,1,2})}{(N_1)_2}\\
      \frac{\me(\nu_{1,1,1}\nu_{1,2,1})}{N_2} & \frac{\me(\nu_{2,1,1}\nu_{2,2,1})}{N_1} & \frac{N_1-1}{N_2}\me(\nu_{1,1,1}\nu_{1,2,1}) & \me(\nu_{1,1,1}\nu_{2,2,1}) & \me(\nu_{1,2,1}\nu_{2,1,1}) & \frac{N_2-1}{N_1}\me(\nu_{2,1,1}\nu_{2,2,2})\\
      \frac{\me(\nu_{1,1,1}\nu_{1,2,1})}{N_2} & \frac{\me(\nu_{2,1,1}\nu_{2,2,1})}{N_1} & \frac{N_1-1}{N_2}\me(\nu_{1,1,2}\nu_{1,2,1}) & \me(\nu_{1,2,1}\nu_{2,1,1}) & \me(\nu_{1,1,1}\nu_{2,2,1}) & \frac{N_2-1}{N_1}\me(\nu_{2,1,1}\nu_{2,2,1})\\
      \frac{N_1\me((\nu_{1,2,1})_2)}{(N_2)_2} & \frac{\me((\nu_{2,2,1})_2)}{N_2-1} & \frac{(N_1)_2}{(N_2)_2}\me(\nu_{1,2,1}\nu_{1,2,2}) & \frac{N_1}{N_2-1}\me(\nu_{1,2,1}\nu_{2,2,1}) & \frac{N_1}{N_2-1}\me(\nu_{1,2,1}\nu_{2,2,1}) & \me(\nu_{2,2,1}\nu_{2,2,2})\\
   \end{array}
   \right)
   \]
   \label{figure3}
\end{figure}
\end{landscape}
\subsubsection{Multi-type exchangeable partition probability function} \label{MEPPF}
Let $E$ be some at most countable set. We call $E$ the type space. For $j=(j_k)_{k\in E}\in\nz_0^E$ let ${\cal T}_j$ denote the set of all tensors $T=(t_{k,\ell})_{k,\ell\in E}$ with vector entries $t_{k,\ell}\in\nz_0^{j_k}$ for all $k,\ell\in E$. Note that if $j_k=0$ then $t_{k,\ell}=()=:\mathbf{0}\in\rz^0$ is the empty vector for all $\ell\in E$. Furthermore, define ${\cal T}:=\bigcup_{j\in\nz_0^E}{\cal T}_j$. The particular (and only) tensor in ${\cal T}_0$ with all its entries $t_{k,\ell}=()$ being the empty vector is denoted by $T_0$. The following definitions are inspired from the properties of the transition functions $\Phi_j$, $j\in\nz_0^E$, of multi-type Cannings models studied in Section \ref{consistency}.
\begin{definition}[normalization]
   A function $p:{\cal T}\to[0,1]$ is called normalized, if $p(T_0)=1$.
\end{definition}
\begin{definition}[consistency]
   A function $p:{\cal T}\to[0,1]$ is called consistent, if for all $j=(j_k)_{k\in E}\in\nz_0^E$ and all tensors $T\in{\cal T}_j$ the equality
   \begin{equation} \label{con}
      p(T)\ =\ \sum_{k\in E} p(T(k,\ell))
      + \sum_{k\in E}\sum_{s=1}^{j_k}p(T(k,\ell,s))
   \end{equation}
   holds for each $\ell\in E$, where the tensor $T(k,\ell)$ is obtained from $T=(t_{k,\ell})_{k,\ell\in E}$ by replacing the (possibly empty) vector $t_{k,\ell}=(i_{k,\ell,1},\ldots,i_{k,\ell,j_k})$ by $(i_{k,\ell,1},\ldots,i_{k,\ell,j_k},1)$ and the (possibly empty) vector $t_{k,\ell'}=(i_{k,\ell',1},\ldots,i_{k,\ell',j_k})$ by $(i_{k,\ell',1},\ldots,i_{k,\ell',j_k},0)$ for all $\ell'\ne\ell$, and the vector $T(k,\ell,s)$ is obtained from $T$ by replacing the single entry $i_{k,\ell,s}$ by $i_{k,\ell,s}+1$.
\end{definition}
\begin{definition}[symmetry]
   A function $p:{\cal T}\to[0,1]$ is called symmetric, if for all $j=(j_k)_{k\in E}\in\nz_0^E$, all tensors $T=(t_{k,\ell})_{k,\ell\in E}\in{\cal T}_j$ and all permutations $\sigma_{k,\ell}\in S_{j_k}$, $k,\ell\in E$, the symmetry relation
   \begin{equation} \label{sym}
      p(T)\ =\ p(\sigma(T))
   \end{equation}
   holds, where $\sigma(T):=(\sigma_{k,\ell}t_{k,\ell})_{k,\ell\in E}$.
\end{definition}
A function $p:{\cal T}\to[0,1]$ is called a \emph{multi-type partition probability function} (M-PPF), if $p$ is normalized and consistent. If $p$ is in addition symmetric, then $p$ is called a \emph{multi-type exchangeable partition probability function} (M-EPPF). For the single-type case $|E|=1$, this terminology is in agreement with the notion of the exchangeable partition probability function (EPPF) of Pitman \cite{Pitman1995}.




\bibliographystyle{plain}
%


\end{document}